\documentclass[reqno]{amsart}
\usepackage{amsmath}
\usepackage{amsfonts}
\usepackage{amssymb}
\usepackage{amscd}
\usepackage{amsthm}
\theoremstyle{plain}
\newtheorem{theorem}{Theorem}[section]
\newtheorem{corollary}[theorem]{Corollary}
\newtheorem{lemma}[theorem]{Lemma}
\newtheorem{proposition}[theorem]{Proposition}
\theoremstyle{definition}
\newtheorem{definition}[theorem]{Definition}
\newtheorem{problem}{Problem}
\newtheorem*{question}{Question}
\newtheorem*{notation}{Notation}
\newtheorem*{claim}{Claim}
\newtheorem*{remark}{Remark}
\newtheorem*{case}{Case}
\newtheorem*{ttheorem}{Theorem}
\numberwithin{equation}{section}
\newcommand{\N}{\mathbb{N}}
\newcommand{\si}{\sigma}
\newcommand{\e}{\varepsilon}
\begin{document}
\title{Function spaces not containing $\ell_{1}$}
\author{S.A. Argyros}
\address[S.A. Argyros]{Department of Mathematics,
National Technical University of Athens}
\email{sargyros@math.ntua.gr}
\author{A. Manoussakis}
\address[A. Manoussakis]{Department of Sciences,
Technical University of Crete } \email{amanouss@science.tuc.gr}
\author{M. Petrakis}
\address[M. Petrakis]{Department of Sciences,
Technical University of Crete } \email{minos@science.tuc.gr}
\begin{abstract}
For $\Omega$ bounded and open subset of $\mathbb{R}^{d_{0}}$ and
$X$ a reflexive Banach space with $1$-symmetric basis, the
function space $JF_{X}(\Omega)$ is defined.
This class of spaces
includes the classical James function space.
Every member of this
class is separable and has non-separable dual.
We provide a proof
of topological nature that $JF_{X}(\Omega)$ does not contain an
isomorphic copy of $\ell_{1}$.
We also investigate the structure
of these spaces and their duals.
\end{abstract}
\maketitle
\tableofcontents
\section*{Introduction.}
The study of  separable Banach spaces not containing $\ell_{1}$
and having a non-separable dual was initialized in the middle of
70's with the fundamental papers of R.C. James \cite{J} and
J.Lindenstrauss-C.Stegall~\cite{ls}, where the first examples of
such spaces were provided.
James example is the widely known James
Tree space $JT$, which is a sequence space.
The space $JT$ is also
investigated in \cite{ls} where additionally a function space
sharing similar properties is defined.
The later is called the
James function space $JF$.
The space $JT$ and its variations have
been studied extensively e.g. \cite{AI}, \cite{A}, \cite{BHO},
\cite{H}.
The most impressive member of this class of spaces has
been provided by W.T. Gowers~\cite{G}.
Mixing the $JT$ structure
with $H.I$ constructions, Gowers was able to present a separable
Banach space not containing $\ell_{1}$ such that every subspace
has non-separable dual.
Thus his space does not contain
$\ell_{1}$, $c_{0}$ or a reflexive subspace.
Examples lying
between $JT$ and $H.I$ spaces are also contained in \cite{AT}.
Non
separable versions of $JT$ are contained in \cite{A}, \cite{f}.
In
the present paper we deal with function spaces related to James
function space, on which the norm is defined as follows:

The space $JF$ is the completion of $L^{1}((0,1))$ endowed with
the following norm
$$
\Vert f\Vert_{JF}= \sup\left\{\left (\sum_{j=1}^{m}
(\int_{I_{j}}fd\mu)^{2}\right)^{1/2} :\{I_{j}\}_{j=1}^{m}\,\,
\text{interval\,partition\,of}\,\,[0,1]\right\}\,\, .
$$
Our goal is to define and study norms extending the above norm.
Thus we consider the following class of spaces.
\\
\indent Let $\Omega$ be a bounded open subset of
$\mathbb{R}^{d_{0}}$. We denote by $\mathcal{P}(\Omega)$ the set
of all open parallelepipeds contained in $\Omega$ (i.e each
$T\in\mathcal{P}(\Omega)$ has the form
$T=\prod_{d=1}^{d_{0}}(\alpha_{d},\beta_{d})$,
$\alpha_{d}<\beta_{d}$ for every $d\leq d_{0}$). For
$(X,\Vert\cdot\Vert_{X})$ a reflexive Banach space with
$1-$symmetric basis $(e_{n})_{n}$ the space $JF_{X}(\Omega)$ is
defined as the completion of $L^{1}(\Omega)$ under the norm,
\begin{equation}\label{e2}
\Vert f\Vert_{JF_{X}(\Omega)}=\sup \left\{\Vert\sum_{j=1}^{m}
\left(\int_{T_{j}}fd\mu\right)e_{j}\Vert_{X}
:\{T_{j}\}_{j=1}^{m}\subset\mathcal{P}(\Omega),\,T_{i}\cap
T_{j}=\emptyset\right\}\,\, .
\end{equation}
Clearly for $\Omega=(0,1)$ and $X=\ell_{2}$, $JF_{X}(\Omega)$ is
the James function space $JF$.
Throughout the paper, by $JF_{X}$
we shall denote the space $JF_{X}((0,1))$.

It is an easy observation that for $f\in L^{1}(\Omega)$,
$\Vert f\Vert_{JF_{X}(\Omega)}\leq\Vert f\Vert_{L^{1}}$, hence
$JF_{X}(\Omega)$ is separable.
It turns out that its dual
$JF^{\ast}_{X}(\Omega)$ is non-separable.
The main result of the
paper is the following :

\smallskip
{\bf Theorem A}. For $\Omega$ and $X$ as before the space
$JF_{X}(\Omega)$ does not contain an isomorphic copy of
$\ell_{1}$.
\smallskip

The proof of this result is of topological nature and it is
different from Lindenstrauss - Stegall's proof for $JF$,
\cite{ls}. Thus our argument leads also to a new proof of this
result in the case of $JF$. S.V. Kisliakov, \cite{K}, has also
provided another proof of the fact that $\ell_{1}$ does not embed
in $JF$. His elegant argument is also of topological nature, and
it uses the representation of $JF$ as a subspace of the space of
functions of bounded $2-$variation.

For $\Omega$ subset of $\mathbb{R}^{d_{0}}$, $d_{0}>1$, our
arguments use  properties of the parallelepipeds and it is not
clear to us, if the non embedding of $\ell_{1}$ holds for norms
which are defined by families of convex sets different than
$\mathcal{P}(\Omega)$.

It is worth noting that the structure of $JF_{X}(\Omega)$ depends
on the geometry of the set $\Omega$.
For example if $\Omega$ is a
finite union of parallelepipeds, then on the positive cone of
$L^{1}(\Omega)$ the $\Vert\cdot\Vert_{1}$ and
$\Vert\cdot\Vert_{JF_{X}(\Omega)}$ are equivalent.
This property
is no longer true if $\Omega$ is the Euclidean ball of
$\mathbb{R}^{d_{0}}$, $d_{0}>1$.
However as we show in the fourth
section, $JF_{X}$ is isomorphic to a complemented subspace of
$JF_{X}(\Omega)$, for any $\Omega$ bounded open subset of
$\mathbb{R}^{d_{0}}$.
Also for $d_{0}>1$, $\Omega$ bounded open
subset of $\mathbb{R}^{d_{0}}$, $JF_{X}(\Omega)$ is not isomorphic
to a subspace of $JF_{X}$.

\noindent Let's now pass to describe how this paper is organized.

The first section is devoted to the proof of Theorem A
mentioned above.
We also show that the formal identity
$I:L^{1}(\Omega)\mapsto JF_{X}(\Omega)$ is a Dunford-Pettis
strictly singular operator.

In the second section we show that the Haar system is a Schauder
basis of $JF_{X}$ and that $X$ is isomorphic to a complemented
subspace of $JF_{X}$.

In the third section  we study the quotients of
$JF_{X}^{*}(\Omega)$.
We prove two results which show an essential
difference between the structures of $JF_{X}(\Omega)$, when
$\Omega$ is either $(0,1)$ or a bounded open subset of
$\mathbb{R}^{d_{0}}$ with $d_0>1$.
In particular we show the
following :

\smallskip
{\bf Proposition B}. Let $\Delta$ be a countable dense subset of
$(0,1)$ and let $Y$ be the closed subspace of $JF_{X}$ generated
by the set $\{\chi_{I}: I$ is an interval with  endpoints in
$\Delta\}$. Then
$$
JF_{X}^{\ast}/Y=X^{\ast}_{\mathbb{R}}\,\, .
$$
Here
$X_{\mathbb{R}}$ denotes the space endowed with the norm induced
on $c_{00}(\mathbb{R})$ by the space $X$.
Clearly
$X_{\mathbb{R}}^{*}$ is a reflexive space.
In particular, in the
$JF$ case, we obtain that
$$
JF^{\ast}/Y\equiv\ell_{2}(\mathbb{R})\,\, .
$$
It is well known that $JT$ shares a similar property.
Next we show
the following.

\smallskip
{\bf Proposition C}. Let $d_{0}>1$, and $\Omega$ be a bounded open
subset of $\mathbb{R}^{d_{0}}$.
Then for every  separable subspace
$Y$ of $JF_{X}^{*}(\Omega)$, the quotient $JF_{X}^{*}(\Omega)/Y$
is not reflexive.

\smallskip
Propositions B and C yield that for $\Omega$ as in Proposition C,
$JF^{\ast}_{X}(\Omega)$ is not a quotient of $JF^{\ast}_{X}$.
Hence $JF_{X}(\Omega)$ does not embed in $JF_{X}$.

In section 4, we prove that $JF_{X}( (0,1)^{d_{1}})$ is isomorphic
to a complemented subspace of $JF_{X}(\Omega)$, where $\Omega$
denotes a bounded open subset of $\mathbb{R}^{d_{0}}$ and $d_{1}<
d_{0}$.

The fifth section is devoted to the isomorphic embedding of
$c_{0}$ in $JF$. It is stated in \cite{ls} that there exists a
subsequence of  Rademacher functions equivalent to the usual basis
of $c_{0}$. This is a peculiar property which has as consequence
that $JF$ is not embedded into $JT$. The later holds since $JT$ is
$\ell_{2}$ saturated, \cite{J}, \cite{AI}. In this section we
characterize those reflexive Banach spaces $X$ with $1-$symmetric
basis such that the Rademacher functions in $JF_{X}$ contain a
subsequence equivalent to $c_{0}$ basis. It turns out that these
spaces must satisfy a property, defined as Convex Combination
Property $(CCP)$. $CCP$ trivially holds on $\ell_{p}$ spaces,
$1<p<\infty$, but not in all spaces with $1-$symmetric basis. For
example the Lorentz space $d(w,p)$, where $w=(\frac{1}{n})_{n}$
and $1<p<\infty$, fails this property. Concerning the $CCP$ we
prove the following.

\smallskip
{\bf Theorem D}. The following are equivalent:
\begin{enumerate}
\item The space $X$ satisfies $CCP$.
\item The normalized sequence
$(\frac{r_{n}}{\Vert r_{n}\Vert})_{n\in\N}$ in $JF_{X}$ of
Rademacher  functions contains a subsequence equivalent to the
usual basis of $c_{0}$.
\end{enumerate}
The last section contains the study of  alternative descriptions
of $JF_{X}$ and $JF_{X}^{**}$.
Namely we introduce the space of
functions of $X-$bounded variation, which is defined as follows.
$$
V_{X}=\{f:\;[0,1]\to \mathbb{R}: f(0)=0, \Vert
f\Vert_{V_{X}}<\infty\}
$$
where
$$ \Vert f\Vert_{V_{X}}=\sup\{
\Vert\sum_{i=1}^{n-1}(f(t_{i+1})-f(t_{i}))e_{i}\Vert_{X}
:\mathcal{P}=\{t_{i}\}_{i=1}^{n}\,\,\,\text{partition
of}\,\,\,[0,1]\}.
$$
We also consider the closed subspace $V_{X}^{0}=\{f\in V_{X}:
\lim_{\delta(\mathcal{P})\to 0}\alpha_{X}(f,\mathcal{P})=0\}$ of
$V_{X}$.
This definition extends the corresponding definition of
$V_{p}$, $V_{p}^{0}$, $1<p<\infty$, appeared in \cite{K} and used
in his proof that $\ell_{1}$ does not embed in $JF_{p}$.
It is not
hard to see that $JF_{X}$ is isometric to $V_{X}^{0}$ and moreover
to show the following.

\smallskip
{\bf Theorem E}. The following hold:
\begin{enumerate}
\item $JF_{X}$ is isometric to $V_{X}^{0}$.
\item $JF_{X}^{**}$ is isometric to $V_{X}$.
\item
On the bounded subsets of $V_{X}^{0}$ the weak topology
coincides
with the topology of pointwise convergence in $C[0,1]$.
\end{enumerate}
The above representations of $JF_{X}$ as $V_{X}^{0}$ and its
second dual as  $V_{X}$ have certain advantages.
For example
$JF_{X}$, as a completion of $L^{1}(0,1)$ is not contained in the
set of measurable functions, while $V_{X}^{0}$ is contained in the
continuous functions $C[0,1]$.
Further $V_{X}$ is a set of Baire-1
functions.
The use of $V_{X}^{0}$, $V_{X}$ appears very useful in
the study of the subspaces of $V_{X}^{0}$.
Indeed we first
investigate the properties of $f\in C[0,1]\cap(V_{X}\setminus
V_{X}^{0})$ and  prove the following.

\smallskip
{\bf Proposition F}. A function $f\in C[0,1]\cap(V_{X}\setminus
V_{X}^{0})$ iff  $f$ is a difference of bounded semicontinuous
functions.
\smallskip

For the proof of this result we make use of some recent results
from descriptive set theory, \cite{KL}, \cite{Ro}.
As a
consequence of the above result we obtain the following theorem.

\smallskip
{\bf Theorem G.}  A subspace $Y$ of $V_{X}^{0}$ contains $c_{0}$
iff $C[0,1]\cap (\overline{Y}^{w^{*}}\setminus Y)\not=\emptyset$.
\smallskip

Furthermore for $X=\ell_{p}$, $1<p<\infty$, we get

\smallskip
{\bf Theorem H.} A non reflexive subspace $Y$ of $V_{p}^{0}$
either contains $c_{0}$ or $\ell_{p}$.
\smallskip

Finally we prove the following.

{\bf Theorem I.} A closed subspace $Y$ of $V_{X}^{0}$ has the
point of continuity property, $(PCP)$ iff
$(\overline{Y}^{w^{*}}\setminus Y)\cap C[0,1]=\emptyset$.

Theorems G and H concern the isomorphic structure of the spaces
$JF_{X}$. This is not  completely clarified even in the case of
$JF$. A detailed study of $JF$ has been provided by S.Buechler's
Ph.D. Thesis \cite{Bu}, where the following results are included.
For all $2\leq p<\infty$ and $\e>0$ the space $\ell_{p}$ is
$(1+\e)-$isomorphic to a subspace of $JF$, and also every
normalized weakly null sequence has an unconditional subsequence.
These two results indicate the richness and the regularity of
$JF$. Moreover it is shown that for $1<p<2$, $\ell_{p}$ is not
isomorphic to a subspace of $JF$.

In the last part of the paper we present some open
problems related to our investigation.
\section{The non embedding of $\ell_{1}$ into $JF_{X}(\Omega)$.}
\label{sl1} This section contains the proof that $\ell_{1}$ does
not embed into $JF_{X}(\Omega)$.
We start with some preliminary
result concerning the structure of these spaces.

We recall, from the introduction, that $JF_{X}(\Omega)$ is defined
for each $\Omega$ bounded open subset of $\mathbb{R}^{d_{0}}$, and
$X$ reflexive space with $1-$symmetric basis, and it is the
completion of $L^{1}(\Omega)$ under the norm described in the
introduction, see (\ref{e2}).
A direct application of the triangle
inequality yields that for $f\in L^{1}(\Omega)$,
$$
\Vert f\Vert_{JF_{X}(\Omega)}\leq\Vert f\Vert_{1}
$$
where $\Vert\cdot\Vert_{1}$ denotes the $L^{1}(\Omega)$ norm.
Hence $JF_{X}(\Omega)$ is a separable Banach space.
Further we
recall that $\mathcal{P}(\Omega)$ denotes the set of all open
parallelepipeds contained in $\Omega$.
Every
$T\in\mathcal{P}(\Omega)$ with $\mu_{d_{0}}(T)>0$, $\mu_{d}$
throughout this paper, denotes Lebesgue measure in
$\mathbb{R}^{d}$, defines a bounded linear functional on
$JF_{X}(\Omega)$ under the rule
$$
L^{1}(\Omega)\ni f\mapsto T^{\ast}(f)=\int_{T}fd\mu
$$
We easily see that $\Vert
T^{\ast}\Vert_{JF_{X}^{\ast}(\Omega)}=1$.
Also it is easy to see
that for $T_{1}, T_{2}\in\mathcal{P}(\Omega)$ with
$T_{1}\not=T_{2}$, $\Vert T_{1}^{\ast}-T_{2}^{\ast}\Vert\geq 1$,
hence the space $JF^{\ast}_{X}(\Omega)$ is non-separable.
\begin{remark} The functional $T^{\ast}$ defined by a
parallelepiped $T$ is the same if $T$ is considered either open or
closed. Hence we shall not distinguish the cases if $T$ is open or
closed.

>From the definition of the norm, it is easy to see that, if
$\Omega_{1}=\prod_{i=1}^{d_{0}}(\alpha_{i},\beta_{i})$ and
$\Omega_{2}=\prod_{i=1}^{d_{0}}(\gamma_{i},\delta_{i})$  then
$JF_{X}(\Omega_{1})$ is $1-$isometric to $JF_{X}(\Omega_{2})$.
Also if $\Omega$ is  open bounded subset of $\mathbb{R}^{d_{0}}$
and $T$ an open parallelepiped contained in $\Omega$, then
$JF_{X}(T)$ is $1-$complemented subspace of $JF_{X}(\Omega)$.
\end{remark}
\begin{notation}
Throughout the paper we denote by $X$ a reflexive Banach space
with 1-symmetric basis $(e_{i})_{i}$, and by
$\Vert\cdot\Vert_{X}$, $\Vert\cdot\Vert_{X^{\ast}}$ the norm in
$X$ and $X^{\ast}$ respectively. The space $X$ satisfies the
property $\lim_{n\to\infty}\Vert \sum_{i=1}^{n}e_{i}\Vert=\infty$
(c.f \cite{LT}).
\end{notation}
The following subset of $JF^{\ast}_{X}(\Omega)$ plays a key role
in the proof of the main result of this section :
$$
\mathcal{S}=\left\{\sum_{n=1}^{k}a_{n}T_{n}^{*} :
\{T_{n}\}_{n=1}^{k}\,\,\text{pairwise disjoint elements
of}\,\,\,\mathcal{P}(\Omega)\,\,\text{and} \,\,\,
\Vert\sum_{n=1}^{k}a_{n}e_{n}^{*} \Vert_{X^{*}}\leq 1\right\}.
$$
Observe that the definition of the norm of $JF_{X}(\Omega)$ yields
that for $\phi\in\mathcal{S}$,
$\Vert\phi\Vert_{JF_{X}^{\ast}(\Omega)}\leq 1$.
\begin{lemma}\label{l1}
The set $\mathcal{S}$ norms isometrically the space
$JF_{X}(\Omega)$.
\end{lemma}
\begin{proof}
Indeed, since $\mathcal{S}$ is a subset of $B_{
JF_{X}^{\ast}(\Omega)}$ we obtain that for $f\in L^{1}(\Omega)$,
$$
\sup\{\langle\phi, f\rangle :\phi\in\mathcal{S}\} \leq\Vert
f\Vert_{JF_{X}(\Omega)}\,\, .
$$
For the converse given $\varepsilon>0$,
$f\in L^{1}(\Omega)$ choose $\{T_{i}\}_{i=1}^{n}$ disjoint
elements of $\mathcal{P}(\Omega)$ such that
$$
\Vert f\Vert_{JF_{X}(\Omega)}-\varepsilon\leq \Vert\sum_{i=1}^{n}
\left(\int_{T_{i}}fd\mu\right)e_{i}\Vert_{X}.
$$
Next choose
$\{\alpha_{i}\}_{i=1}^{n}$ such that
$\Vert\sum_{i=1}^{n}\alpha_{i}e_{i}^{\ast}\Vert_{X^{\ast}}=1$
and
$$
\sum_{i=1}^{n}\alpha_{i}\int_{T_{i}}fd\mu =
\Vert\sum_{i=1}^{n}\left(\int_{T_{i}}fd\mu\right)e_{i}\Vert_{X}.
$$
Clearly setting $\phi=\sum_{i=1}^{n}\alpha_{i}T_{i}^{\ast}$ we
obtain that $\phi\in\mathcal{S}$ and
$$
\Vert f\Vert_{JF_{X}(\Omega)}-\varepsilon\leq\langle
\phi,f\rangle\,\,.
$$
\end{proof}
As a consequence we obtain the following
\begin{lemma}\label{l2}
The set  $\overline{\mathcal{S}}^{w^{\ast}}$ contains the extreme
points of $B_{JF_{X}^{\ast}(\Omega)}$. Hence
$B_{JF_{X}^{\ast}(\Omega)}$ =
$\overline{co}^{w^{\ast}}(\mathcal{S})$.
\end{lemma}
\begin{proof}
Assume on the contrary, that there exists an extreme point
$x^{*}\in B_{JF_{X}^{*}}$ with
$x^{*}\not\in\overline{\mathcal{S}}^{w^{*}}$. Since the $w^{*}-$
slices of $x^{*}$ define a neighborhood basis for the
$w^{*}$-topology, there exists a slice $S(x^{*},f,t)$ disjoint
from $\mathcal{S}$. We may assume further that $f$ is linear
combination of characteristic functions and $\Vert f\Vert=1$. From
the above we get that
\begin{enumerate}
\item For some
$\varepsilon>0$, $x^{*}(f)>\sup\{w^{\ast}(f) :w^*\in\mathcal{S}\}
+\varepsilon $.
\item There exists
$\{T_{j}\}_{j=1}^{n}\subset\mathcal{P}(\Omega)$ such that
$$
1=\Vert f \Vert_{JF_{X}(\Omega)}\leq \Vert \sum_{j=1}^{n}
\left(\int_{T_{j}}fd\mu\right)e_{j}\Vert_{X}
+\frac{\varepsilon}{2} \,\, .
$$
\end{enumerate}
Let $x^{\ast}=\sum_{j=1}^{n}\beta_{j}e_{j}^{\ast}\in B_{X^{\ast}}$
such that $x^{\ast}(\sum_{j=1}^{n}
\left(\int_{T_{j}}fd\mu\right)e_{j})= \Vert \sum_{j=1}^{n}
\left(\int_{T_{j}}fd\mu\right)e_{j}\Vert$. We set
$w^{\ast}=\sum_{j=1}^{n}b_{j}T_{j}^{*}$ and observe that
$w^{\ast}\in\mathcal{S}$ and also
$$
w^{\ast}(f)=\Vert \sum_{j=1}^{n} \left(
\int_{T_{j}}fd\mu\right)e_{j}\Vert\,\,.
$$
Summing up all the above we obtain
$$
1-\frac{\varepsilon}{2}\leq\left<f,w^{\ast}\right>
<x^{*}(f)-\varepsilon\leq 1-\varepsilon\,\, ,
$$
a contradiction, and the proof is complete.
\end{proof}
\begin{lemma}\label{l3}
Let $(T_{n})_{n\in \mathbb{N}}$ be a subset of
$\mathcal{P}(\Omega)$ and suppose that
$w^{*}-\lim_{n\to\infty}T_{n}^{*}=s^{\ast}$.
Then either
$s^{\ast}$ is equal to the characteristic function of a
parallelepiped $T$ with $T\in\mathcal{P}(\Omega)$ and
$T_{n}\rightarrow T$ pointwise or $s^{\ast}=0$ and
$\mu_{d_{0}}(T_{n})\rightarrow 0.$
\end{lemma}
\begin{proof}
Choose any subsequence $(T_{n_{k}})_{k}$ such that
$\overline{T}_{n_{k}}\rightarrow S$ pointwise.
If
$T^{o}\not=\emptyset$ then $ S^{o}\in\mathcal{P}(\Omega)$, and
$\lim_{k\to\infty}\mu(T_{n_{k}}\bigtriangleup S)=0$.
Set $S^{o}=T$.
It follows that $\int_{T_{n_{k}}}\psi \rightarrow
\int_{T}\psi, \forall\psi\in L^{1}(\Omega)$, and hence
$\int_{}\psi s^{\ast }=\int_{T}\psi$ for all $\psi\in
L^{1}(\Omega)$, which yields that $s^{\ast}=T^{*}.$
If
$S^{o}=\emptyset$ then it is easy to see that
$\mu_{d_{0}}(T_{n})\rightarrow 0$ and
$T_{n}^{\ast}\stackrel{w^{\ast}}{\longrightarrow} 0\,.$
\end{proof}
\noindent As a consequence we obtain the following corollary.
\begin{corollary}\label{l4}
Let $n_{0}\in \mathbb{N}$ be fixed,
$s_{n}^{\ast}=\sum_{i=1}^{n_{0}}a_{i,n}T_{i,n}^{*}$ with
$\Vert\sum_{i=1}^{n_{0}}\alpha_{i,n}e_{i}^{*}\Vert_{X^{\ast}}
\leq 1.$
Assume that $w^{*}-\lim_{n\to\infty}s_{n}^{\ast}=s^{\ast }.$
Then there exist $(a_{i})_{i}^{n}$ in $\mathbb{R}$ and disjoint
parallelepipeds $\{T_{i}\}_{i=1}^{n_{0}}$ such that
$\Vert\sum_{i=1}^{n_{0}}\alpha_{i}e_{i}^{*}\Vert_{X^{\ast}}
\leq 1$ and $s^{\ast }=\sum_{i=1}^{n_{0}}a_{i}T_{i}^{*}.$
\end{corollary}
We define
$$
V=\{\sum_{n}\alpha_{n}e^{*}_{n}\in B_{X^{*}}:
\{\vert\alpha_{n}\vert\}_{n}\,\,\text{is in decreasing order}\}.
$$
It is easy to see that $V$ is a $w-$compact subset of
$B_{X^{*}}.$
\begin{lemma}\label{l5}
Let $s^{*}_{n}\in \mathcal{S}$ and
$w^{*}-\lim_{n\to\infty}s^{*}_{n}= s^{\ast }.$ Then $ s^{\ast }$
is of the form $\sum_{i=1}^{\infty }a_{i}T_{i}^{*},$ with $\Vert
\sum_{i=1}^{\infty }a_{i}e_{i}^{*}\Vert\leq 1$ and
$\{T_{i}\}_{i}\subset\mathcal{P}(\Omega)$ pairwise
disjoint.
Moreover
\begin{equation}\label{e11}
\Vert\sum_{i}a_{i}T_{i}^{*}\Vert\leq
\Vert\sum_{i}a_{i}e_{i}^{*}\Vert_{X^{*}}\,\, ,
\end{equation}
and hence
$\mathcal{\bar{S}}^{\Vert\cdot\Vert}=\mathcal{\bar{S}}^{w^{*}}.$
\end{lemma}
\begin{proof}
Let $s^{*}_{n}=\sum_{i=1}^{k_{n}}a_{i,n}T_{i,n}^{*}$, where
$\{T_{i,n}\}_{i=1}^{k_{n}}$ are pairwise disjoint and
$\Vert\sum_{i=1}^{k_{n}}a_{i,n}e_{i}^{*}\Vert\leq 1.$ From the
previous corollary we may assume that $\lim_{n}k_{n}=\infty$, and
also that for each $n\in\mathbb{N}$, $\vert a_{i,n}\vert\geq\vert
a_{i+1,n}\vert $ for $i=1,2,\ldots,k_{n}-1.$ Hence the sequence
$(\sum_{i=1}^{k_{n}}a_{i,n}e_{i}^{*})_{n}\subset V.$ Since $V$ is
$w-$compact, by passing to a subsequence if it is necessary we
assume that there exists a
$x^{*}=\sum_{i=1}^{\infty}\alpha_{i}e_{i}^{*}\in V$ such that
$w-\lim_{n\to\infty}\sum_{i=1}^{k_{n}}a_{i,n}e_{i}^{*}=x^{*}.$ We
may also assume that there exists a sequence $\{T_{i}\}_{i}$ of
disjoint parallelepipeds in $\mathcal{P}(\Omega)$ such that
$\lim_{n\to\infty}\mu_{d_{0}}(T_{i,n}\triangle T_{i})=0$, for
every $i=1,2,\ldots$ We set
$u^{*}=\sum_{i=1}^{\infty}\alpha_{i}T_{i}^{*}.$ Since $x^{*}\in
B_{X^{*}}$ we have that $u^{*}\in B_{JF_{X}^{*}(\Omega)}.$
\begin{claim}
$w^{*}-\lim_{n\to\infty}s^{\ast }_{n}= u^{*}$\,\,.
\end{claim}
\noindent {\it Proof of the Claim}. Let $\varepsilon >0$ and
$f=\chi_{T}$, $T$ a parallelepiped in $\mathcal{P}(\Omega).$
Since for every $n\in\mathbb{N}$ the sequence $\{a_{i,n}\}_{i}$
is decreasing and the space $X$ is reflexive, we obtain that there
exists $i_{0}\in\mathbb{N}$ such that $\vert
a_{i,n}\vert<\varepsilon$ for every $i\geq i_{0}$ and every
$n\in\mathbb{N}.$ We may also assume that $\Vert
\sum_{i>i_{0}}a_{i}e_{i}^{*}\Vert <\varepsilon.$ Let
$n_{0}\in\mathbb{N}$ such that for every $n\geq n_{0}$, $k_{n}\geq
i_{0}$, $\vert a_{i,n}-\alpha_{i}\vert
\leq\frac{\varepsilon}{i_{0}}$ and
$\mu_{d_{0}}(T_{i,n}\bigtriangleup T_{i})
\leq\frac{\varepsilon\Vert f\Vert}{i_{0}}$ for every
$i=1,2,\ldots,i_{0}.$ Then for every $n\geq n_{0}$ we have that,
\begin{align*}
\vert &\sum_{i=1}^{k_{n}}a_{i,n}T_{i,n}^{*}(\chi_{T})-
\sum_{i=1}^{\infty}a_{i}T_{i}^{*}(\chi_{T})\vert
\\
&\leq\vert\sum_{i=1}^{i_{0}}a_{i,n}T_{i,n}^{*}(\chi_{T})
-\sum_{i=1}^{i_{0}}a_{i}T_{i}^{*}(\chi_{T})\vert +
\vert\sum_{i=i_{0}+1}^{k_{n}}a_{i,n}T_{i,n}^{*}(\chi_{T})\vert
+\sum_{i=i_{0}+1}^{\infty}a_{i}T_{i}^{*}(\chi_{T})\vert
\\
&\leq\vert\sum_{i=1}^{i_{0}}a_{i,n}T_{i,n}^{*}(\chi_{T})-
\sum_{i=1}^{i_{0}}a_{i}T_{i}^{*}(\chi_{T})\vert+\varepsilon
\sum_{i=i_{0}+1}^{k_{n}}\mu_{d_{0}}(T_{i,n}\cap T) +
\Vert\sum_{i=i_{0}+1}^{\infty}a_{i}T_{i}^{*}\Vert\,
\Vert\chi_{T}\Vert
\\
&\leq\varepsilon\Vert\chi_{T}\Vert + \varepsilon\Vert\chi_{T}\Vert
+\varepsilon\Vert\chi_{T}\Vert +\varepsilon\Vert\chi_{T}\Vert\leq
4\varepsilon\Vert\chi_{T}\Vert\,\,.
\end{align*}
Since the linear span of the characteristic of parallelepipeds in
$\Omega$ is a dense subset of $JF_{X}(\Omega)$, we have the
result. Inequality~(\ref{e11}) follows directly from the
definition of the norming set $\mathcal{S}$ and immediately
implies that $\mathcal{\bar{S}}^{\Vert\cdot\Vert}=
\mathcal{\bar{S}}^{w^{*}}.$
\end{proof}
\begin{corollary}\label{l6}
Let $(f_{n})_{n}$ be a bounded sequence in $JF_{X}(\Omega).$
Then $(f_{n})_{n}$ is weakly Cauchy iff for every
$T\in\mathcal{P}(\Omega)$ the numerical sequence $T^{*}(f_{n})$
is convergent.
\end{corollary}
\begin{proof}
Assume that $(T^{*}(f_{n}))_{n}$ is convergent for all
$T\in\mathcal{P}(\Omega).$ Then the same remains valid for every
$x^{*}$ in the norm closure of the linear span of the set
$\{T^{*}: T\in\mathcal{P}(\Omega)\}.$ The later subspace of
$JF_{X}^{*}(\Omega)$ contains
$\mathcal{\bar{S}}^{\Vert\cdot\Vert}$ which from, Lemma~\ref{l5},
coincides with $\mathcal{\bar{S}}^{w^{*}}.$ Lemma~\ref{l2} yields
that  $\mathcal{\bar{S}}^{w^{*}}$ contains the extreme points of
the ball of $JF_{X}^{*}(\Omega)$, and the result is obtained from
Rainwater theorem \cite{rain}. The other direction is obvious.
\end{proof}
We pass now to prove that $\ell_{1}\not\hookrightarrow
JF_{X}(\Omega)$ for every $\Omega$ open and bounded subset of
$\mathbb{R}^{d_{0}}$, and every reflexive Banach space $X$ with
1-symmetric basis. Lemmas \ref{l10}, \ref{l11} will be the main
ingredients for the proof. For a given sequence $(f_{n})_{n}$,
using these lemmas we choose constructively a $w-Cauchy$
subsequence. As we have mentioned in the introduction, it is not
clear to us, if we can have corresponding results for norms
defined by convex subsets different from parallelepipeds.
\begin{notation}Let $T=\Pi_{d=1}^{d_{0}}(\alpha_{d},\beta_{d})$
be an element of $\mathcal{P}(\Omega).$ In the sequel for
$d\leq d_{0}$
we denote by $m_{d}(T)$ the number $\beta_{d}-\alpha_{d}.$
\end{notation}
Let $(f_{n})_{n}$ be a normalized sequence in $JF_{X}(\Omega)\cap
L_{1}(\Omega).$ We shall show that $(f_{n})_{n}$ contains a
$w-Cauchy$ subsequence. We start with the following lemma.
\begin{lemma}\label{l10}
For every  $\varepsilon> 0$, there exists
$n(\varepsilon)\in\mathbb{N}$ such that for every $1\leq d \leq
d_{0} $, $L\in[\mathbb{N}]$ and $\delta>0$ there exist
$L^{\prime}\in [L]$ and $n(\varepsilon)$ disjoint parallelepipeds
$T_{\varepsilon,d,1}^{\delta}$,\ldots,$T_{\varepsilon,d,
n(\varepsilon)}^{\delta}$ with
$m_{d}(T_{\varepsilon,d,i}^{\delta})<\delta$ for all $i\leq
n(\varepsilon)$, satisfying the following property:

For each $d\leq d_{0}$ and every parallelepiped $T\in
\mathcal{P}(\Omega)$ which is disjoint from
$\{T_{\varepsilon,d,i}^{\delta}\}_{i=1}^{n(\varepsilon)}$, with
the property $m_{d}(T)<\delta$ we have that
$$
\limsup_{n\in
L^{\prime}}\vert\int_{T}f_{n}d\mu\vert\leq\varepsilon \,\,.
$$
\end{lemma}
In the above lemma, when $d_{0}=1$ we require the measure of the
interval be arbitrarily small. In  higher dimensions, for our
consideration, it is not sufficient that $\mu(T)$ be arbitrarily
small, but we require that $m_{d}(T)$ be arbitrarily small for
each of the sides of the parallelepiped. This is due mainly to the
geometry of parallelepipeds of higher dimensions.
\begin{proof}
On the contrary, suppose that the conclusion does not hold. Then
there exists $\varepsilon_{0}>0$ such that for every $n\in
\mathbb{N}$ there exist
$d_{n}\leq d_{0}$, $L(n)\in [\mathbb{N}]$
and $\delta_{n}>0$
such that for every
$L^{\prime}\in [L(n)]$ and disjoint parallelepipeds
$T_{1},T_{2},\ldots,T_{n}$, with $m_{d_{n}}(T_{j} )<\delta_{n}$,
$j\leq n$, there exists a parallelepiped $T$ disjoint from
$T_{j}$, $j\leq n$,  with $m_{d}(T)<\delta_{n}$ and
$$
\limsup_{n\in L^{\prime}}\vert\int_{T}f_{n}d\mu\vert
>\varepsilon_{0}.
$$
Let $m\in\mathbb{N}$ be such that
$\varepsilon_{0}\Vert\sum_{i=1}^{m}e_{i}\Vert >1.$ We set
$d=d_{m}$, $\delta=\delta_{m}$ and $L=L(m).$ Next we inductively
choose disjoint parallelepipeds $\{S_{j}\}_{j=1}^{m}$ with
$m_{d}(S_{j})<\delta$ and $L_{1}\supset L_{2}\supset\ldots\supset
L_{m}$ of $L$, such that for every $1\leq j\leq m$ and $n\in
L_{j}$, $\vert\int_{S_{j}}f_{n}d\mu\vert\geq\varepsilon_{0}.$ The
choice goes as follows:

We choose $S_{1}\in\mathcal{P}(\Omega)$ with
$m_{d}(S_{1})<\delta.$ From our assumption
$\limsup_{L}\vert\int_{S_{1}}f_{n}\vert>\varepsilon_{0}$, hence
there exists $L_{1}\subset L$ such that
$\vert\int_{S_{1}}f_{n}\vert\geq\varepsilon_{0}.$ This completes
the choice of $S_{1}$ and $L_{1}.$ Assume that
$S_{1},\ldots,S_{j-1}$, $L_{1}\supset\ldots\supset L_{j-1}$,
$j\leq m$, have been chosen satisfying the inductive assumption.
Since $j-1<m$ and $L_{j-1}\subset L$, there exists
$S_{j}\in\mathcal{P}(\Omega)$ disjoint from $S_{1},\ldots,S_{j-1}$
with $m_{d}(S_{j})<\delta$ and
$\limsup_{L_{j-1}}\vert\int_{S_{j}}f_{n}\vert>\varepsilon_{0}.$ We
choose $L_{j}\subset L_{j-1}$ such that
$\vert\int_{S_{j}}f_{n}\vert\geq\varepsilon_{0}$ for all $n\in
L_{j}$ and this completes the inductive construction.

Take any $n\in L_{m}.$ Then $n\in L_{j}$ for all $1\leq j\leq m$
and hence, $\vert\int_{L_{j}}f_{n}d\mu\vert\geq\varepsilon_{0}.$
Hence
$$
\Vert f_{n}\Vert\geq
\Vert\sum_{j=1}^{m}\left(\int_{S_{j}}f_{n}d\mu\right)e_{j}\Vert
\geq\varepsilon_{0}\Vert\sum_{j=1}^{m}e_{j}\Vert>1\,\,,
$$
which contradicts our assumption that $\Vert f_{n}\Vert=1.$
\end{proof}
{\bf Selection of a $w-Cauchy$ subsequence.}

To obtain the desired $w-Cauchy$ subsequence of $(f_{n})_{n}$ we
shall apply repeatedly Lemma~\ref{l10} in the following manner:

Fix $\varepsilon=\frac{1}{k}$, and $1\leq d\leq d_{0}.$ Then
Lemma~\ref{l10} yields that there exists $n(k)$ such that for
every $L\in[\mathbb{N}]$ and every $\delta=\frac{1}{m}$, a finite
sequence $\{T_{k,d,1}^{\frac{1}{m}},\ldots,
T_{k,d,n(\varepsilon)}^{\frac{1}{m}}\}$ and
$L_{m}^{d}(\varepsilon)\subset L$ are defined so that the
conclusion of Lemma~\ref{l10} is fulfilled.

\noindent Therefore for a fixed $k$ we inductively define
sequences $\{L_{m}^{d}(k)\}_{m}$, $\{T_{k,d,1}^{\frac{1}{m}},$
$\ldots,T_{k,d,n(k)}^{\frac{1}{m}}\}_{m}$ such that
\begin{enumerate}
\item
$\{L_{m}^{d}(k)\}_{m}$ is a decreasing sequence of subsets of
$\mathbb{N}$.
\item For each $m\in\mathbb{N}$ the pair $
\{T_{k,d,1}^{\frac{1}{m}},\ldots,T_{k,d,n(k)}^{\frac{1}{m}}\}$ and
$L_{m}^{d}(k)$ satisfies the conclusion of Lemma~\ref{l10} for
$\delta=\frac{1}{m}$ i.e.
\begin{align}\label{e12}
\text{for\,every}\,\,\,T\in\mathcal{P}(\Omega)\,\,\text{disjoint
from}&\,\,\,\, \{T_{k,d,i}^{\frac{1}{m}}\}_{i\leq n(k)}\,\,\,
\text{and}\,\,\,m_{d}(T)<\frac{1}{m}\,\,
\mathrm{we\,have}\notag\\
&\limsup_{n\in
L_{m}^{d}(k)}\vert\int_{T}f_{n}d\mu\vert\leq\frac{1}{k}\,\,.
\end{align}
\end{enumerate}
Notice that for a fixed $k$ and $d\leq d_{0}$,
$\lim_{m\to\infty}m_{d}(T_{k,d,i}^{\frac{1}{m}})=0$ for every
$i\leq n(k).$ Hence passing to a subsequence we may assume that
there exist $n(k)$ faces $H_{k,d,1}$,$\ldots$,$H_{k,d,n(k)}$, each
one of the form
$$
\prod_{i=1}^{d_{0}}[\alpha_{i},\beta_{i}]\,\,\,
\mathrm{with}\,\,\,\alpha_{d}=\beta_{d}\,\,,
$$
and moreover
\begin{equation}\label{e13}
\lim_{m\to\infty}
dist_{H}(H_{k,d,i},T_{k,d,i}^{\frac{1}{m}})=0\,\,\,\,\,
\mathrm{for\,\ all} \,\,i\leq n(k)\,\,,
\end{equation}
where $dist_{H}$ denotes the Hausdorff distance.
We set $A_{k,d}$
be the set of all coordinates of the extremes points of the
faces $H_{k,d,1}$,$\ldots,H_{k,d,n(k)}.$
Clearly $A_{k,d}$ is a finite set.

Let $L_{\infty}^{d}(k)$ denote any diagonal set of the
decreasing sequence $\{L_{m}^{d}(k)\}_{m}.$
Applying the above procedure inductively we obtain
\begin{flalign}\label{e14}
L_{\infty}^{1}(k)\supset\ldots\supset
L_{\infty}^{d_{0}}(k)=L_{\infty}(k)\,\,,
\notag \\
\{T_{k,d,1}^{\frac{1}{m}},\ldots,
T_{k,d,n(k)}^{\frac{1}{m}}\}_{m\in\mathbb{N}},\,\,1\leq d\leq d_{0}\,,
\\
\{A_{k,d}\},\,\,1\leq d \leq d_{0},\,\, A_{k,d}\,\, \mathrm{is\,
finite\, subset\,of}\,\, \mathbb{R}\,,
\notag
\end{flalign}
such that for every $d\leq d_{0}$, $L_{\infty}^{d}(k)$,\,
$\{T_{k,d,1}^{\frac{1}{m}},$
$\ldots,T_{k,d,n(k)}^{\frac{1}{m}}\}_{m\in\mathbb{N}}$ satisfies
(\ref{e12}), and $A_{k,d}$ is defined as above.

Proceeding now by induction for $k=1,2,\ldots$ we  choose
$\{L_{\infty}(k)\}_{k}$ a decreasing sequence of infinite sets,
$\{\{T_{k,d,1}^{\frac{1}{m}},\ldots
,T_{k,d,n(k)}^{\frac{1}{m}}\}_{m\in\mathbb{N},d\leq d_{0}}\}_{k} $
and $\{\,\{A_{k,d}\}_{d=1}^{d_{0}}\}_{k}$ such that for any
$k\in\mathbb{N}$, the corresponding families satisfies
(\ref{e14}). Set
$$ F=\cup_{k\in\mathbb{N}}
\cup_{d=1}^{d_{0}}A_{k,d}\cup\mathbb{Q}\,.
$$
Clearly $F$ is a countable set and hence the set $M$ of the
parallelepipeds in $\mathcal{P}(\Omega)$ with vertices in
$F^{d_{0}}$ is also countable. Therefore
there exists a diagonal subset $L$ of
$\{L_{\infty}(k)\}_{k}$ such that
$$
\lim_{n\in L}\int_{T}f_{n}d\mu\,\,\,\,\,\,\text{exists for
every}\,\,T\in M\,\,.
$$
Our intention is to show that $(f_{n})_{n\in L}$ is $w-Cauchy.$
This follows from the next lemma.
\begin{lemma}\label{l11} For any $T$ in $\mathcal{P}(\Omega)$
$$
\lim_{n\in L}\int_{T}f_{n}d\mu\,\,\,\,\,\,\text{exists\,.}
$$
\end{lemma}
\begin{proof}
Let $k\in\mathbb{N}.$ It is enough to show that
$$
\limsup_{n\in L}\int_{T}f_{n}d\mu - \liminf_{n\in
L}\int_{T}f_{n}d\mu <\frac{ 4d_{0} }{k}\,\,.
$$
Let $T=\{(x_{1},x_{2},\ldots,x_{d_{0}} ):\, \alpha_{d}< x_{d}<
\beta_{d},\, \forall d\leq d_{0}\}$ and for $d\leq d_{0}$,
$$
\Pi_{d}^{1}=\{(x_{1},x_{2},\ldots,x_{n_{0}} )\in T
:\,x_{d}=\alpha_{d}\,\,\mathrm{and}\,\, \alpha_{j}< x_{j}<
\beta_{j},\,\,\mathrm{ for}\,\, j\not= d \}\, ,
$$
$$
\Pi_{d}^{2}=\{(x_{1},x_{2},\ldots,x_{n_{0}} )\in T
:\,x_{d}=\beta_{d}\,\,\mathrm{and}\,\,\alpha_{j}\leq x_{j}\leq
\beta_{j},\,\,\mathrm{for}\,\, j\not= d\}\,\,.
$$
We set
\begin{center}$
I_{1}=\{1\leq d\leq d_{0} : \{a_{d},b_{d}\}\not\subset F \}.$
\end{center}
We assume that $I_{1}$ is not empty, otherwise $T$ belongs in
$\mathcal{M}$ and therefore $\lim_{n\in L}\int_{T}f_{n}d\mu$
exists. For every $d\leq d_{0}$ with $d\in I_{1}$,\,
$\#\{a_{d},b_{d}\}\cap F\geq 1.$
For simplicity we assume that for every
$d\in I_{1}$, $\#\{a_{d},b_{d}\}\cap F=2.$
The proof of the general case follows similar arguments.

Hence, we assume that for every $d\in I_{1}$,
$$
\{\alpha_{d},\beta_{d}\}\cap \{pr_{d}(H_{k,d,i}),\,\,i\leq
n(k)\}=\emptyset \,,
$$
where $pr_{d}$ denotes the $d-projection$ in
$\mathbb{R}^{d_{0}}.$

For every $d\in I_{1},$ we set
$$
\delta_{d}^{1}=\min\{ dist\Bigl( a_{d}, pr_{d}( H_{k,d,i})\Bigl) :
i\leq n(k)\}>0\,,
$$
and
$$
\delta_{d}^{2}=\min\{dist\Bigl( b_{d}, pr_{d}(H_{k,d,i})\Bigl) :
i\leq n(k)\}>0 \,.
$$
Set $\delta_{d}=\min\{\delta_{d}^{1},\delta_{d}^{2}\}$, and
$\delta_{0}=\frac{1}{2} \min\{ \delta_{d} : d\in I_{1} \}>0$. We
choose $m\in\mathbb{N}$ such that
$\frac{1}{m}<\frac{\delta_{0}}{10}$ and
$dist_{H}(H_{k,d,i},\,T_{k,d,i}^{\frac{1}{m}})
<\frac{\delta_{0}}{10},$ for every $i\leq n(k)$ and every $d\in
I_{1}.$

For each  $d\in I_{1}$   we choose $p_{d},q_{d}\in F$ such that
\begin{center}
$0<p_{d}-\alpha_{d} <\frac{1}{m}$ \,\,\, and \,\,\,
$0<\beta_{d}-q_{d}<\frac{1}{m}.$
\end{center}
For  $d\in I_{1}$, we consider the following parallelepipeds
\begin{center}
$S_{d}^{1}=\{(x_{1},\ldots,x_{d_{0}})\in T : \alpha_{d}<
x_{d}<p_{d} $ and $ \alpha_{j}< x_{j}<\beta_{j}$ for $j\not=
d\}$\,\,
\end{center}
\begin{center}
$S_{d}^{2}=\{(x_{1},\ldots,x_{d_{0}})\in T : q_{d}<x_{d}<
\beta_{d} $ and $ \alpha_{j}< x_{j}<\beta_{j}$ for $j\not=
d\}$\,\,.
\end{center}
We observe that $m_{d}(S_{d}^{i})<\frac{1}{m}$, for $i=1,2$ and
$d\in I_{1}.$ Furthermore  $S_{d}^{i}$ is disjoint from the
elements of the set $\{ T^{\frac{1}{m}}_{k,d,i},\,\,  i\leq
n(k)\}.$ Clearly the above two properties are also satisfied by
any parallelepiped $R$ which is contained in $S_{d}^{i}.$ The
properties of $\{T_{k,d,i}^{\frac{1}{m}}\}_{i=1}^{n(k)}$, yield
that for every parallelepiped $R$ contained in $S_{d}^{i}$ we have
that
\begin{equation}\label{e15}
\limsup_{n\in L}\int_{R}f_{n}d\mu- \liminf_{n\in
L}\int_{R}f_{n}d\mu\leq \frac{2}{k}\,.
\end{equation}
Let $K$ be the parallelepiped
$$
K=\{(x_{1},\ldots,x_{n_{0} })\in T : p_{d}\leq x_{d}\leq q_{d},
\mathrm{for}\,\,  d\in
I_{1},\,\,\mathrm{otherwise}\,\,\alpha_{d}<x_{d}<\beta_{d} \}\,.
$$
Clearly the parallelepiped $K$ has vertices in $F^{d_{0}}$ and
hence
\begin{equation}\label{e16}
\limsup_{n\in L}\int_{K}f_{n}d\mu - \liminf_{n\in
L}\int_{K}f_{n}d\mu =0\,.
\end{equation}
For every $d\in I_{1}$ let,
\begin{align*}
T_{d}^{1}=\{ (x_{1},\ldots,x_{d_{0} })\in T:
x_{j}\in pr_{d}(K)\,\, \text{for\, every }\,\, j<d,\\
\alpha_{d}< x_{d}<p_{d}\,\,\mathrm{and}\,\, \alpha_{j}<
x_{j}<\beta_{j}\,\,\text{for}\,\, j>d\}\,
\\
\intertext{and} T_{d}^{2}=\{ (x_{1},\ldots,x_{d_{0}})\in T :
x_{j}\in pr_{d}(K)\,\,\text{for\, every }\,\, j<d,\,\,
\\
q_{d}<x_{d}< \beta_{d}\,\,\mathrm{and}\,\, \alpha_{j}<
x_{j}<\beta_{j}\,\,\text{for}\,\, j>d\}\,.
\end{align*}
The following properties are easily established.
\begin{enumerate}
\item For $d_{1},d_{2}\in I_{1}$, $i,j\in\{1,2\}$ such that
$T_{d_{1}}^{i}\not= T_{d_{2}}^{j}$, we have that
$\mu_{d_{0}}(T_{d_{1}}^{i}\cap T_{d_{2}}^{j})=0.$
\item $T=\cup_{d\in I_{1}}(T_{d}^{1}\cup T_{d}^{2})\cup K$.
\item $T_{d}^{i}$ is contained in $S_{d}^{i}$ for every $d\in
I_{1}$, $i=1,2$.
\end{enumerate}
\noindent Then  $(3)$ and (\ref{e15}) yield
\begin{equation}\label{e17}
\limsup_{n\in L}\int_{T_{d}^{i}} f_{n}d\mu
-\liminf\int_{T_{d}^{i}}f_{n}d\mu \leq \frac{2}{k}\,\,.
\end{equation}
Finally  $(1)$,$(2)$, (\ref{e16}) and (\ref{e17}) yield
$$
\limsup_{n\in L}\int_{T} f_{n}d\mu -\liminf\int_{T}f_{n}d\mu \leq
\frac{4d_{0} }{k}\,.
$$
The proof is complete.
\end{proof}
\begin{theorem}\label{l112} Let $\Omega$ be an open bounded
subset of $\mathbb{R}^{d_{0}}$ and $X$ be a reflexive
Banach space with symmetric basis.
Then $\ell_{1}$ does not embed into $JF_{X}(\Omega).$
\end{theorem}
\begin{proof}
It is enough to show that every normalized sequence $(f_{n})_{n}$
in $JF_{X}(\Omega)$ has a $w-$Cauchy subsequence. A perturbation
argument yields that $(f_{n})_{n}$ could be assumed to belong to
$L^{1}(\mu).$ Lemmas \ref{l10} and \ref{l11}  yield that it
contains a $w-$Cauchy subsequence.
\end{proof}

In \cite{pe} it has been shown that the identity $I:
L^{1}([0,1])\mapsto JF_{X}$ is a strongly regular operator. This
result is naturally extended to the corresponding
$I:L^{1}(\Omega)\mapsto JF_{X}(\Omega).$ The proof of it uses
martingale techniques. A rather simple argument yields that $I$ is
Dunford - Pettis operator (i.e maps weakly compact sets to norm
compact) a property weaker than the strong regularity. For  sake
of completeness we include a proof of it. The proof uses the
following
\begin{lemma}\cite{pe}\label{pe1}
Let $X$ be a Banach space with $1-$symmetric basis, not containing
$\ell_{1}.$ Then for each $\varepsilon>0$ there exists $\delta>0$
such that for every $\{\alpha_{i}\}_{i=1}^{k}$ with
$\sum_{i=1}^{k}\vert\alpha_{i}\vert\leq 1$ and
$\max\{\vert\alpha_{i}\vert: i\leq k\}\leq\delta$ we have
$$
\Vert\sum_{i=1}^{k}\alpha_{i}e_{i}\Vert<\varepsilon\,.
$$
\end{lemma}
\begin{proposition}\label{l114}
Let $(f_{n})_{n}$ be a weakly null sequence in $L^{1}(\Omega).$
Then $\Vert f_{n}\Vert_{JF_{X}(\Omega)}\to 0$.
\end{proposition}
\begin{proof}
Assume on the contrary that there exists a normalized
weakly null sequence $(f_{n})_{n}$ in $L^{1}(\Omega)$
so that for all $n\in\mathbb{N}$
\begin{center}
$\Vert f_{n}\Vert_{JF_{X}(\Omega)}>\varepsilon>0$.
\end{center}
For each $f_{n}$ choose a family
$\{T_{n,j}\}_{j=1}^{k_{n}}\in\mathcal{P}(\Omega)$ such that
\begin{equation}\label{e18}
\Vert\sum_{j=1}^{k_{n}}
\left(\int_{T_{n,j}}f_{n}d\mu\right)e_{j}\Vert_{X}>
\frac{\varepsilon}{2}\,.
\end{equation}
Further we observe that
\begin{equation}\label{e19}
\sum_{j=1}^{k_{n}}\vert\int_{T_{n,j}}f_{n}d\mu\vert\leq 1\,.
\end{equation}
Hence Lemma~\ref{pe1} yields that there exists $\delta >0$ such
that for each $n\in\mathbb{N}$ there exists $j_{n}$ satisfying
$$
\vert\int_{T_{n,j_{n}}}f_{n}d\mu\vert >\delta\,.
$$
Passing, if it is necessary, to a subsequence we may assume that
$T_{n,j_{n}}\to T.$

The uniform integrability of $(f_{n})_{n}$ guarantees that
$\lim_{n\to\infty}\vert\int_{T} f_{n}d\mu\vert>\frac{\delta}{2}$,
which contradicts the weak convergence of $(f_{n})_{n}$ to zero.
\end{proof}
\begin{corollary}\label{115}
The identity $I:L^{1}(\Omega)\mapsto JF_{X}(\Omega)$ is
Dunford-Pettis and strictly singular operator.
\end{corollary}
\begin{proof} Proposition~\ref{l114} yields immediately
that the identity is a Dunford - Pettis operator.
The strict singularity follows from a well known property of
$L_{1}(\Omega)$, namely  every
closed subspace $Z$ of it  either is reflexive or contains
$\ell_{1}.$ Since the identity is  D-P operator  we obtain that
it is not isomorphism on any reflexive subspace of
$L_{1}(\Omega)$, while Theorem~\ref{l112} yield that the
identity is not isomorphism on any subspace containing
$\ell_{1}.$ Hence it is strictly singular.
\end{proof}
\section{The Haar system in $JF_{X}$.}\label{sHaar}
In this section we give a simple proof that the Haar system
$(h_{n})_{n}$ is a basis for $JF_{X}.$ We also prove that $X$ is
isomorphic to a complemented subspace of $JF_{X}.$
Our approach is based on elementary properties of
symmetric sequences.
\begin{notation} For
every $Q=\{I_{j}\}_{j=1}^{m}$ partition of $[0,1]$, we set
$$
\tau(Q,f)= \Vert\sum_{j=1}^{m}(\int_{I_{j}}f)e_{j}\Vert_{X}
\,\,.
$$
\end{notation}
\begin{proposition}\label{p21}
Let $f\in JF_{X}$ and
$\mathcal{P}_{0}=\{t_{0},t_{1},\ldots,t_{n}\}$
be a partition of $[0,1]$
such that  $f|_{(t_{i-1},t_{i})}=\alpha_{i}$ for every
$i=1,2,\ldots,n.$
Then for every
partition $Q=\{I_{j}\}_{j=1}^{m}$ , of $[0,1]$ into disjoint
intervals, there exists a partition
$\mathcal{P}\subseteq\mathcal{P}_{0}$ such that $\tau
(Q,f)\leq\tau(\mathcal{P},f).$
\end{proposition}
This result is very useful for computing norms of functions in
$JF_{X}$ and studying the structure of the space. One immediate
consequence is that the Haar system is a Schauder basis for
$JF_{X}.$

To give an idea of the proof consider a partition
$P_{0}=\{0=t_{0}<t_{1}<t_{2}<t_{3}=1\}$ of $[0,1]$ into three
subintervals, $T_{i}=[t_{i-1},t_{i}]$, $i=1,2,3.$ Assume that the
function $f$ on $[0,1]$ takes value $\alpha_{i}$ on $T_{i}$,
$i\leq 3$, and $\alpha_{1},\alpha_{3}\geq 0$ and $\alpha_{2}\leq
0.$ One can check that
$$
\Vert f\Vert_{JF}=\max\{\vert\int_{0}^{1}f\vert,\Vert
(\int_{T_{1}}f)e_{1}+(\int_{T_{2}}f)e_{2}+
\int_{T_{3}}f)e_{3}\Vert_{X}\}\,\,.
$$
The proof for the general case
uses finite induction to replace any partition $Q$ of $[0,1]$ by a
partition $Q^{\prime}\subset \mathcal{P}_{0}$, so that $\tau
(f,Q)\leq\tau (f,Q^{\prime}).$

This  proposition will give us as corollary the result of this
section. In the proof we shall use the following that restates a
result from \cite{pe}.
\begin{lemma}\label{l22}
Let $X$ be a Banach space with a $1-$symmetric basis
$(e_{i})_{i\in\mathbb{N}}$ and $x=\sum_{n=1}^{k}\alpha_{n}e_{n}$ a
linear combinations  with $\alpha_{n}\geq 0$. Then for every
$\{\sigma_{j}\}_{j=1}^{\ell}$disjoint partition of
$\{1,\ldots,n\}$ we have that
$$
\Vert x\Vert_{X}\leq\Vert\sum_{j=1}^{\ell}
\left(\sum_{n\in\sigma_{j}}\alpha_{n}\right)e_{j}\Vert_{X}\,\,.
$$
\end{lemma}
\begin{lemma}\label{l23}
Let $X$ be a Banach space with 1-symmetric basis $(e_{i})_{i}.$
Let $x=\sum_{i=1}^{j-1}\beta_{i}e_{i}+ (\beta_{j}
+\lambda_{j}\alpha)e_{j}+ \lambda_{j+1}\alpha e_{j+1}+
\sum_{i=j+2}^{r}\beta_{i}e_{i}$, where
$\lambda_{j},\lambda_{j+1}\geq 0$ and
$\lambda_{j}+\lambda_{j+1}\leq 1.$ Then
$$
\Vert x\Vert\leq
\begin{cases}
\Vert \sum_{i=1}^{j-1}\beta_{i}e_{i}+
(\beta_{j}+(\lambda_{j}+\lambda_{j+1})\alpha e_{j}
+\sum_{i=j+2}^{r}\beta_{i}e_{i}\Vert & if\,\alpha\cdot\beta_{j} >0\\
\\
\Vert \sum_{i=1}^{j-1}\beta_{i}e_{i}+ \beta_{j}e_{j}+
(\lambda_{j}+\lambda_{j+1})\alpha e_{j+1}
+\sum_{i=j+2}^{r}\beta_{i}e_{i}\Vert & if\,\alpha\cdot\beta_{j}
<0\,\,.
\end{cases}
$$
\end{lemma}
\begin{proof}
Indeed, if $\beta_{j}\cdot \alpha >0$ using the Lemma~\ref{l22}
for the block $(\beta_{j}+\lambda_{j}\alpha)e_{j}+
\lambda_{j+1}\alpha e_{j+1}$, we have the inequality. If
$\beta_{j}\cdot \alpha <0$, we may assume that $\alpha >0$ since
the basis is 1-unconditional. Then $\vert
\beta_{j}+\lambda_{j}\alpha\vert \leq
\max\{\vert\beta_{j}\vert,\lambda_{j}\alpha\}.$ Substituting the
coefficient of $e_{j}$ by the coefficient
$\max\{\vert\beta_{j}\vert, \lambda_{j}\alpha\}$ and adding the
term $\min\{\vert\beta_{j}\vert,\lambda_{j}\alpha\}e_{m}$ and
reordering if  necessary, we have that the norm will increase, due
to the symmetric property of the basis. Using Lemma~\ref{l22} we
have the result.
\end{proof}
\begin{proof}[Proof of Proposition~\ref{p21}]
We  use finite induction in order to replace the partition
$Q=\{I_{j}\}$ by a partition $\mathcal{P}=\{S_{l}\}$ with
endpoints in $\mathcal{P}_{0}$ such that $\tau(Q,f)\leq \tau
(\mathcal{P},f).$ In each step of the induction we shall replace
some of the intervals $I_{j}$ by  appropriate intervals $S_{j}.$

For every partition $P=\{R_{j}\}_{j}$ we may assume that $\max
R_{j-1}=\min R_{j},$ for every $j$, otherwise we add the interval
$(\max R_{j-1},\min R_{j})$ in the partition, and we have that
$\tau(P,f)$ increases, due the symmetric property of the basis of
$X$.

For every $i=1,2,\ldots,n$ we set $T_{i}=[t_{i-1},t_{i}].$ For
every $i=1,2,\ldots,n$, let $A_{i}=\{1\leq j\leq m: I_{j}\subseteq
T_{i}\}.$ We also set $B=\{1\leq j\leq m: I_{j}\cap
T_{i}\not=\emptyset\,\,\, \text{for at least two}\,\,i's,\, 1\leq
i\leq n\}.$ Let
$A_{i}=\{j_{i_{1}},j_{i_{1}}+1,\ldots,j_{i_{1}}+r_{i}\}\not
=\emptyset.$ From Lemma~\ref{l22} we have that
\begin{align*}
\tau(Q,f) \leq \Vert
\sum_{j\in B}(\int_{I_{j}}f)e_{j}+
\sum_{i: A_{i}\not=\emptyset}\sum_{j\in A_{i}} (
\alpha_{i}\mu(\cup_{j\in A_{i}} I_{j})e_{j_{i_{1}}} \Vert\,\,.
\end{align*}
For every $i$ such that $A_{i}\not=\emptyset$ we replace
$\cup_{j\in A_{i}}I_{j}$ by the maximal interval contained in
$T_{i}$ and it is disjoint with $I_{j}$, $j\in B.$ We have that
$\tau(f,Q)$ increases, due the symmetric property of the basis. In
the sequel, from the above observation, we shall assume that if
two intervals, in the inductive construction, intersect an
interval $T_{i}$, at least one of them intersects another interval
$T_{j}$ as well.

Assume that we have replaced the partition $\{I_{j}\}_{j}$ by a
partition $\{S_{1},\ldots,S_{l-1}, I_{l}, I_{l+1},$
$\ldots,I_{m}\}$ which increases $\tau(Q,f)$, such that the
intervals $S_{i}$, $i\leq l-1$, have endpoints in
$\mathcal{P}_{0}$ and also that we have replaced
$I_{l}$ by  an interval such
that the initial point belongs to $\mathcal{P}_{0}.$ Let
$t_{j_{l}}\in\mathcal{P}_{0}$ be such that $t_{j_{l}-1}< \max
I_{l}<t_{j_{l}}.$ We distinguish two cases.

\smallskip
\noindent {\it Case 1.}  $I_{l}\subset
T_{j_{l}}=[t_{j_{l}-1},t_{j_{l}}].$

We have assumed that $I_{l+1}$ is not contained in $T_{j_{l}}.$
>From the hypothesis for $f$ we have that
$$
\int_{I_{l+1} }f= \int_{I_{l+1}\cap  T_{j_{l}} }f+ \int_{I_{l+1}
\setminus T_{j_{l}}}f= \alpha_{j_{l}} \mu(I_{l+1}\cap
T_{j_{l}})+\int_{I_{l+1}\setminus T_{j_{l}} }f\,\,.
$$
We have two
subcases:

\noindent {\it Subcase 1a.}  $(\int_{I_{l+1}\setminus T_{j_{l}}}f)
\cdot\alpha_{j_{l}} <0$.

Then, from Lemma~\ref{l23}, for the block
$(\int_{I_{l}}f)e_{l}+(\int_{I_{l+1}}f)e_{l+1}$, and the inductive
hypothesis, replacing $I_{l}$ by $S_{l}=T_{j_{l}}\supset
I_{l}\cup(I_{l+1}\cap T_{j_{l}})$ we have that
\begin{align*}
\tau(Q,f) & \leq \Vert \sum_{j=1}^{l-1}(\int_{S_{j} } f)e_{j}+
\left(\int_{
S_{l} }f \right)e_{l} +(\int_{I_{l+1}\setminus T_{j_{l}}}
f)e_{l+1}+\sum_{j> l+1}\left(\int_{I_{j} }f\right)e_{j}
\Vert_{X}\,\,.
\end{align*}
We also replace the interval $I_{l+1}$ by the interval
$S_{l+1}=I_{l+1}\setminus T_{j_{l}}$ which has initial point in
$\mathcal{P}_{0}.$

\noindent {\it Subcase 1b.} $(\int_{I_{l+1} }f)\cdot
\alpha_{j_{l}} >0.$

In this case, using Lemma~\ref{l23}, we can replace the interval
$I_{l}$ by the interval $S_{l}=[\min T_{j_{l}},\max I_{l+1}],$
which has initial point in $\mathcal{P}_{0}$ and we delete the
interval $I_{l+1}$ and therefore the coefficient of $e_{l+1}.$ The
norm increases, since
$$
\left| \alpha_{j_{l}}\mu(I_{l})+ \alpha_{j_{l}}\mu(I_{l+1}\cap
T_{j_{l}})+\int_{I_{l+1}\setminus T_{j_{l}}}f\right| \leq \left|
\alpha_{j_{l}}\mu(T_{j_{l}})+\int_{I_{l+1} \setminus
T_{j_{l}}}f\right| =\left\vert\int_{S_{l}}f\right|\, .
$$
For the interval $[\min T_{j_{l}}, \max I_{l+1}],$ it could be the
case that $\max I_{l+1}\not\in\mathcal{P}_{0}.$ If such a case
occurs, then in case 2 we show how to replace it.

\smallskip
\noindent {\it Case 2.} $I_{l}\varsubsetneqq
T_{j_{l}}=[t_{j_{l}-1},t_{j_{l}}]$.
\newline
Then the interval $I_{l}$ intersects an interval $T_{i}$ for some
$i<j_{l}$, since $\max I_{l}<t_{j_{l}}.$ From the hypothesis for
$f$ we have that
\begin{center}
$ \int_{I_{l} } f=\int_{I_{l}\setminus
T_{j_{l}}}f+\alpha_{j_{l}}\mu(I_{l}\cap T_{j_{l}}) $.
\end{center}
We have two subcases.

\smallskip
\noindent {\it Subcase 2a.} $(\int_{I_{l}\setminus T_{ j_{l} }
}f)\cdot \alpha_{j_{l}} >0.$
\newline
By our assumptions we have that If $I_{l+1}\cap\not=\emptyset.$
Using Lemma~\ref{l22}, for the block $(\int_{I_{l}}f)e_{l}+
(\int_{I_{l+1}}f)e_{l+1}$ we may assume that $I_{l+1}\subsetneqq
T_{j_{l}}.$ From the hypothesis for $f$ we have
\begin{center}
$\int_{I_{l+1}}f= \alpha_{j_{l}}\mu(I_{l+1}\cap T_{j_{l}})+
\int_{I_{l+1}\setminus T_{j_{l}}}f $.
\end{center}
If $\alpha_{j_{l}}\cdot (\int_{I_{l+1}\setminus T_{j_{l}}}f)>0$,
applying  Lemma~\ref{l22} for the block
$(\int_{I_{l}}f)e_{l}+(\int_{I_{l+1}}f)e_{l+1}$ we get that,
\begin{align*}
\tau(Q,f) \leq \Vert \sum_{j=1}^{ l-1}\left(\int_{ S_{j}
}f\right)e_{j} +\left(\int_{I_{l} \cup T_{j_{l}}\cup I_{l+1}
}f\right)e_{l}+ \sum_{j\geq l+2}\left(\int_{I_{j}
}f\right)e_{j}\Vert_{X}\,\,.
\end{align*}
We have replaced the intervals $I_{l}$, $I_{l+1}$ by the interval
$S_{l}=I_{l} \cup T_{j_{l}}\cup I_{l+1}$ which has initial point
in $\mathcal{P}_{0}.$

If $(\int_{I_{l+1}\setminus T_{j_{l}}}f)\cdot \alpha_{j_{l}}<0$
using the symmetric property of the basis and Lemma~\ref{l23},
replacing $I_{l}$ by $S_{l}=[\min I_{l},\max T_{j_{l}}]$, we get
that,
\begin{align*}
\tau(Q,f) \leq & \Vert\sum_{j=1}^{l-1}(\int_{S_{j}}f)e_{j}+
(\int_{S_{l}}f)e_{l}+(\int_{I_{l+1}\setminus T_{j_{l}}}f)e_{l+1}+
\sum_{j>l+1}(\int_{I_{j}}f)e_{j}\Vert_{X}\,\,.
\end{align*}
We also replace the interval $I_{l+1}$ by the interval
$S_{l+1}=I_{l+1}\setminus T_{j_{l}}$ which has initial point in
$\mathcal{P}_{0}.$

\noindent {\it Subcase 2b.} $ (\int_{I_{l}\setminus T_{j_{l}}
}f)\cdot \alpha_{j_{l}} <0$\,.

In this case we  replace $I_{l}$ by $S_{l}=I_{l}\setminus
T_{j_{l}}$, which has endpoints in $\mathcal{P}_{0}$, and we add
the term $(\int_{I_{l}\setminus T_{j_{l}}}f)e_{l+1}$, transferring
the sum $\sum_{j\geq l+1}(\int_{I_{j} }f)e_{j}.$ The interval
$I_{l}\setminus T_{j_{l}}$ has initial point in $\mathcal{P}_{0}$,
and we follow the arguments of Case $1$ for the interval
$I_{l}\setminus T_{j_{l}}.$

Following the above arguments for the intervals which we get in
the above cases, with initial point in $\mathcal{P}_{0}$,
we have the result.
\end{proof}
Let us recall the definition of the Haar system $(h_{n})_{n}.$
We set $h_{1}=\chi_{[0,1]}$ and
$$
h_{2^{k}+i}=
\chi_{[\frac{2i-2}{2^{k+1}},\frac{2i-1}{2^{k+1}}]}-
\chi_{(\frac{2i-1}{2^{k+1}},\frac{2i}{2^{k+1}}]}\,\,\text{for
every}\,\, 1\leq i\leq 2^{k}, k=0,1,\ldots
$$
\begin{corollary}\label{c24}
The Haar system is a Schauder basis for $JF_{X}$.
\end{corollary}
\begin{proof}
It is enough to show that $\Vert
\sum_{i=1}^{n}a_{i}h_{i}\Vert_{JF_{X}}\leq
\Vert\sum_{i=1}^{n+1}a_{i}h_{i}\Vert_{JF_{X}}.$ Set
$\mathcal{P}=(I_{j})_{j=1}^{m}$ the partition corresponding to the
simple function $f=\sum_{i=1}^{n}a_{i}h_{i}$, by
Proposition~\ref{p21}. Standard properties of the Haar system
yield that ${\rm supp}h_{n+1}$ is contained in some $I_{j}.$
Choose a subset $Q$ of $\mathcal{P}$ such that $\tau(f,Q)=\Vert
f\Vert_{JF_{X}}.$ Clearly $\tau(f,Q)=\tau(g,Q)$, where
$g=\sum_{i=1}^{n+1}a_{i}h_{i}$, and this completes the proof.
\end{proof}
{\it Remark.}\, S.Bellenot \cite{be}, has provided a proof that
the Haar system is a basis for $JF$. His proof is based on the
notion a neighborly basis, introduced by R.C.James.
\begin{lemma}\label{l25}
Let $ A_{n}$ be a sequence of successive intervals such that
$\mu(A_{2n-1})=\mu (A_{2n})$ for every $n\in\mathbb{N}.$ We set
$y_{n}=\chi_{A_{2n-1}}-\chi_{A_{2n}}$ for every $n\in\mathbb{N}.$
Then we have that
$$
\Vert\sum_{n}\alpha_{n}e_{n}\Vert_{X}\leq
\Vert\sum_{n}\alpha_{n}\frac{y_{n}}{\mu(A_{2n-1})}\Vert_{JF_{X}}
\leq 2\Vert\sum_{n}\alpha_{n}e_{n}\Vert_{X}\,\,.
$$
Hence $(\frac{
y_{n} }{ \mu(A_{2n-1}) })_{n}$ is $2$-equivalent to the unit
vector basis $(e_{n})_{n}$ of $X.$
\end{lemma}
\begin{proof}
For the left inequality we consider the partition
$(A_{2n-1})_{n}.$ Then
\begin{align*}
\Vert\sum_{n}\alpha_{n}\frac{ y_{n} }{ \mu(A_{2n-1})
}\Vert_{JF_{X}} \geq \Vert\sum_{k}(\int_{A_{2k-1}}
\sum_{n}\alpha_{n}\frac{y_{n}}{\mu(A_{2n-1})})e_{k}\Vert_{X} \geq
\Vert \sum_{n}\alpha_{n}e_{n}\Vert_{X}.
\end{align*}
For the right inequality, let $(I_{j})_{j}$ be any partition of
$[0,1].$ The function
$\sum_{n}\alpha_{n}\frac{y_{n}}{\mu(A_{2n-1})}$ satisfies the
assumptions of Proposition~\ref{p21}, so we may assume that each
of the intervals $I_{j}$ is a finite union of successive $A_{n}.$
It easy to see that for each interval $I_{j}$, we have the
following estimates
\begin{equation*}
\left|\int_{I_{j}}\sum_{n}\alpha_{n}\frac{y_{n}}{\mu(A_{2n-1})}
\right| \leq
\begin{cases}
0   & \text{if}\,\,\,I_{j}=[A_{2k-1}, A_{2m}]\,\,\,,k\leq m\\
\vert\alpha_{m}\vert &  \text{if}\,\,\, I_{j}=[A_{2k-1},
A_{2m-1}]\,\,\,,k\leq m \\
\vert\alpha_{k}\vert  &
\text{if}\,\,\, I_{j}=[A_{2k}, A_{2m}]\,\,\,,k\leq m \\
\vert-\alpha_{k}+\alpha_{m} \vert  & \text{if}\,\,\,
I_{j}=[A_{2k},A_{2m-1}]\,\,\,,k< m
\end{cases}
\end{equation*}
Since the intervals are successive, we have that each $\alpha_{n}$
appears at most two times, and therefore
\begin{equation}\label{last}
\Vert\sum_{n}\alpha_{n}\frac{y_{n}}{\mu(A_{2n-1})}\Vert_{JF_{X}}
\leq 2\Vert\sum_{n}\alpha_{n}e_{n}\Vert_{X}\,\,.
\end{equation}
\end{proof}
\begin{notation}{\rm In the sequel we denote by
$\langle A\rangle$ the linear subspace generated by a subset $A$
of a normed space $Y.$ }
\end{notation}
\begin{theorem}\label{xinjf}
$X$ is isomorphic to a complemented subspace of $JF_{X}.$
\end{theorem}
\begin{proof}
Let $(y_{n})_{n}$ be the sequence defined in the previous lemma.
We prove that the space generated by this sequence is a
complemented subspace of $JF_{X}.$ Consider the map
\begin{align*}
P: JF_{X}\mapsto & \overline{\langle \{y_{n}/\mu(A_{2n-1})
:n\in N\}\rangle}\,\,\,\text{defined by the rule}\\
& f\mapsto \sum_{n}A_{2n-1}^{*}(f)\frac{y_{n}}{\mu(A_{2n-1})}\,\,.
\end{align*}
>From the definition of the map $T$ we have that
$A_{2n-1}^{*}(\frac{y_{k}}{\mu(A_{2k-1})})=\delta_{n,k}$, and from
inequality~(\ref{last})  we have that
\begin{align*}
\Vert\sum_{n}A^{\ast}_{2n-1}(f)
\frac{y_{n}}{\mu(A_{2n-1})}\Vert_{JF_{X}}
&\leq 2 \Vert \sum_{n}A^{\ast}_{2n-1}(f)e_{n}\Vert_{X}\\
&= 2\Vert\sum_{n}(\int_{A_{2n-1}}f)e_{n}\Vert_{X}\leq 2\Vert
f\Vert_{JF_{X}}\,\,.
\end{align*}
It follows that $P$ is a projection with $\Vert P\Vert\leq 2.$
\end{proof}
\begin{remark} It is clear that we can choose subsequences
of the Haar system which fulfill the assumptions of
Lemma~\ref{l25}, and therefore are weakly null. However Haar
system $(h_{n})_{n}$ is not a weakly null sequence. We describe
subsequences of the Haar system which does not converges weakly to
$0$.

Consider $\beta_{1}\in\mathbb{N}$ and $k_{1}\in\mathbb{N}$. For
$n\geq 2$ set $\beta_{n}=8\beta_{n-1}+2$, and set
$$x_{n}=\chi_{[\frac{\beta_{n}}{2^{k_{1}}8^{n-1}},
\frac{\beta_{n}+1}{2^{k_{1}}8^{n-1}}]}-
\chi_{(\frac{\beta_{n}+1}{2^{k_{1}}8^{n-1}},
\frac{\beta_{n}+2}{2^{k_{1}}8^{n-1}}]} \,\,\,\text{for}\,\,\,
n\geq 1.$$ The following are easily established.
\begin{enumerate}
\item The sequence
$(\frac{\beta_{n}}{2^{k_{1}}8^{n-1}})_{n}$ is increasing, while
the sequences $(\frac{\beta_{n}+1}{2^{k_{1}}8^{n-1}})_{n}$,
$(\frac{\beta_{n}+2}{2^{k_{1}}8^{n-1}})_{n}$ are decreasing.
\item $\lim_{n}\frac{\beta_{n}}{2^{k_{1}}8^{n-1}}=
\frac{\beta_{1}}{2^{k_{1}}}+\frac{2}{7\cdot 2^{k_{1}}}
\equiv\beta_{0}.$
\item $\frac{\beta_{n}+1}{2^{k_{1}}8^{n-1}}-\beta_{0}\leq
3(\beta_{0}-\frac{\beta_{n}}{2^{k_{1}}8^{n-1}})$\,\, for every
$n\in\mathbb{N}$.
\end{enumerate}
Set $I=[\frac{\beta_{1}}{2^{k_{1}}},\beta_{0}].$ From the above
properties it follows that
$$\int_{I}\frac{x_{n}}
{\mu([\frac{\beta_{n}}{2^{k_{1}}8^{n-1}},
\frac{\beta_{n}+1}{2^{k_{1}}8^{n-1}}])}= \frac{\mu([
\frac{\beta_{n}}{2^{k_{1}}8^{n-1}},\beta_{0}])}
{\mu([\frac{\beta_{n}}{2^{k_{1}}8^{n-1}},
\frac{\beta_{n}+1}{2^{k_{1}}8^{n-1}}])}\geq \frac14\,\,.
$$
If $\frac{\beta_{1}}{2^{k_{1}}}$ is the initial point of a Haar
function, then the sequence $(x_{n})_{n}$ is a subsequence of the
Haar system, and $(\frac{x_{n}}{\Vert x_{n}\Vert_{JF_{X}}})_{n}$
does not converges weakly to $0$ in $JF_{X}$. On the other hand if
$ (h_{n})_{n\in M}$ is a subsequence of the Haar system such that
$\max{\rm supp}h_{n}=\alpha$ for all $n\in M$, it is not hard to
see that the there exists a subsequence $(h_{n})_{n\in L}$ of
$(h_{n})_{n\in M}$ such that $(\frac{h_{n}}{\Vert
h_{n}\Vert})_{n\in L}$ is weakly null, and in particular is
equivalent to the unit vector basis of $X.$
\end{remark}
%
%
\section{Quotients of $JF_{X}^{\ast }(\Omega)$.}\label{sdual}
This section is devoted to the study of quotients
$JF^{\ast}_{X}(\Omega)/Y$, where $Y$ is a separable subspace of
$JF_{X}^{*}(\Omega)$.
\begin{definition}
Let $X$ be a Banach space with a symmetric basis, and $\Gamma$
an infinite set. We denote by $X_{\Gamma}$ the completion of
$c_{00}(\Gamma)$ under the norm
$$
\Vert
\sum_{i=1}^{n}\alpha_{i}e_{\gamma_{i}}\Vert=
\Vert\sum_{i=1}^{n}\alpha_{i}e_{i}\Vert_{X},\,\,
\mathrm{where\,}\,\,
\gamma_{i}\not=\gamma_{j}\,\,\mathrm{for}\,\,i\not= j\,\,.
$$
\end{definition}
\noindent It is obvious that $X_{\Gamma}$ is reflexive iff
$X$ is reflexive.

We prove the following
\begin{theorem}\label{t32}
Let $\Delta$ be any countable dense subset of $[0,1]$, and
$\{I_{i}\}_{i=1}^{\infty}$ be an enumeration of the subintervals
of $(0,1)$ with endpoints in $\Delta.$ We set
$Y=\overline{\langle\{I_{j}^{\ast} :j=1,2,\ldots\}\rangle}.$ The
quotient space $JF_{X}^{*}/Y$ is isomorphic to $X^{*}_{\Gamma}$,
where the set $\Gamma$ has the cardinality of the continuum.
\end{theorem}
Before  passing  to the proof of the  theorem we make some
preliminary observations. We denote by $\widehat{x^{\ast}}$, the
equivalence class of the functional $x^{\ast}\in JF_{X}^{\ast}.$
Since $\ell_{1}$ does not embed into $JF_{X}$, we have that
$B_{JF_{X}^{*}}=\overline{co}(ExtB_{JF_{X}^{\ast}})$, \cite{Had},
hence Lemmas~\ref{l2}, \ref{l5} yields that
$JF_{X}^{*}=\overline{\langle\{I^{\ast}:\, I=(\alpha,\beta)\subset
(0,1)\}\rangle}.$ Therefore $\langle\{\widehat{I^{\ast}}:\,
I=(\alpha,\beta)\subset (0,1)\}\rangle$ is dense in
$JF_{X}^{\ast}/Y.$
\begin{lemma}\label{l33}
The $\langle\{\widehat{I^{\ast}}:\, I=(\alpha,\delta)
:\,\alpha\not\in\Delta, \delta\in\Delta \}\rangle$ is dense in
$JF_{X}^{\ast}/Y.$
\end{lemma}
\begin{proof}
For simplicity we assume that $\{0,1\}\subset\Delta.$ First we
observe that for $I=(\alpha,\beta)$, $\widehat{I^{\ast}}\not=0$
iff $\{\alpha,\beta\}\nsubseteq \Delta $ and also since $\Delta$
is a dense subset of $(0,1)$ we obtain for any
$I=(\alpha,\beta)$, $I^{\ast}=I_{1}^{\ast}+I_{2}^{\ast}$, where
$I_{1}=(\alpha,\delta)$, $I_{2}=(\delta,\beta)$ with some
$\delta\in\Delta.$ To complete the proof we observe that for
$I_{1}=(\delta,\beta)$, $I_{2}=(\beta,\delta^{\prime})$,
$\delta,\delta^{\prime}\in\Delta$  we have that
$-I_{1}^{\ast}\in\widehat{I_{2}^{\ast}}.$
\end{proof}
\begin{proof}[Proof of Theorem~\ref{t32}]
Let $\alpha_{1}<\alpha_{2}<\ldots<\alpha_{n}$ be such that
$\alpha_{i}\not\in\Delta$ for $i=1,2,\ldots,n$ and $
\widehat{S_{1}^{\ast}}=
\widehat{\chi^{*}}_{(\alpha_{1},\delta_{1})},\ldots,
\widehat{S_{n}^{\ast}}=
\widehat{\chi^{*}}_{(\alpha_{n},\delta_{n})}$, be  elements  of
$JF_{X}^{\ast}/Y$ with $\delta_{i}\in\Delta$, $i\leq n.$

We shall prove that
$$
\frac{1}{2}
\Vert\sum_{j=1}^{n}\lambda_{j}e_{j}^{*}\Vert_{X^{\ast}} \leq
\Vert\sum_{j=1}^{n}\lambda_{j}
\widehat{S_{j}^{*}}\Vert_{JF_{X}^{*}/Y} \leq
\Vert\sum_{j=1}^{n}\lambda_{j}e_{j}^{*}\Vert_{X^{\ast}}\,\,,
$$
for every finite sequence $(\lambda_{j})_{j}$ of reals. This
together with Lemma~\ref{l33} implies that $JF_{X}^{\ast}/Y$ is
isomorphic to $X^{\ast}(\Gamma)$, where
$\Gamma=\{\alpha\in(0,1)\setminus\Delta\}.$ Let's observe that we
may assume that
$\alpha_{1}<\delta_{1}<\alpha_{2}<\delta_{2}<\ldots.$

Let $\{I_{j}^{*} :I_{j}=(d_{1},d_{2}) :d_{1},d_{2}\in\Delta\}$ be
the set of intervals generating the subspace $Y.$ We set
$Y_{m}=\langle \{I_{i}^{*}\}_{i=1}^{m}\rangle.$ Since
$\cup_{m=1}^{\infty}Y_{m}$ is dense in $Y$ we obtain that
$$
dist(\sum_{j=1}^{n}\lambda_{j}S_{j}^{\ast}, Y)= \lim_{m\to\infty}
dist(\sum_{j=1}^{n}\lambda_{j}S_{j}^{\ast}, Y_{m})\,\,.
$$
For $m\in\mathbb{N}$ and $j=1,\ldots,n$,  there exists
$q_{j}^{m}\in (0,1)$ such that the interval
$(\alpha_{j}-q_{j}^{m},\alpha_{j}+q_{j}^{m})$ has the following
properties:
\begin{enumerate}
\item For every $i\leq m$, either
$(\alpha_{j}-q_{j}^{m},\alpha_{j}+q_{j}^{m})\subset I_{i}$ or
$(\alpha_{j}-q_{j}^{m},\alpha_{j}+q_{j}^{m})\cap I_{i}=
\emptyset.$

\smallskip

\item
$(\alpha_{j},\alpha_{j}+q_{j}^{m})
\subset(\alpha_{j},\delta_{j}).$
\end{enumerate}
Set $Q_{j,m}^{1}=(\alpha_{j}-q_{j}^{m},\alpha_{j})$ and
$Q_{j,m}^{2}=(\alpha_{j},\alpha_{j}+q_{j}^{m})$ and choose
$\{\beta_{j}\}_{j=1}^{n}$ such that
\begin{equation}\label{e31}
\Vert\sum_{j=1}^{n}\beta_{j}e_{j}\Vert_{X}=1,\,\,\,
\Vert\sum_{j=1}^{n}\lambda_{j}e_{j}^{\ast}\Vert_{X^{*}}=
\sum_{j=1}^{n}\lambda_{j}\beta_{j}\,\,.
\end{equation}
Set finally
$$
f=\sum_{j=1}^{n}\beta_{j}
\frac{\chi_{Q_{j,m}^{2}}-\chi_{Q_{j,m}^{1}}}{\mu(Q_{j,m}^{2})}\,\,.
$$
Lemma~\ref{l25} yields that
\begin{equation}\label{e32}
\Vert f\Vert_{JF_{X}}\leq 2
\end{equation}
Also, for every $1\leq i\leq m$ property $(1)$ yields that
$\langle I_{i}^{\ast},f\rangle=0.$ Hence (\ref{e31}),(\ref{e32})
implies
$$
dist(\sum_{j=1}^{n}\lambda_{j}S_{j}^{\ast},
Y_{m})\geq\frac12\langle
\sum_{j=1}^{n}\lambda_{j}S_{j}^{\ast},f\rangle
\geq\frac12\sum_{j=1}^{n}\lambda_{j}\beta_{j}= \frac12
\Vert\sum_{j=1}^{n}\lambda_{j}e_{j}^{\ast}\Vert_{X^{\ast}}\,\,.
$$
This proves the left inequality, namely
$$
\frac12\Vert\sum_{j=1}^{n}
\lambda_{j}e_{j}^{\ast}\Vert_{X^{\ast}}\leq
\Vert\sum_{j=1}^{n}\lambda_{j}
\widehat{S_{j}^{\ast}}\Vert_{JF_{X}^{*}/Y}\,\,.
$$
The right inequality follows immediately from the disjointness
of
$\{S_{j}\}_{j=1}^{m}$, since
$$
\Vert\sum_{j=1}^{n}\lambda_{j}
\widehat{S_{j}^{\ast}}\Vert_{JF_{X}^{*}/Y} \leq
\Vert\sum_{j=1}^{n}\lambda_{j}S_{j}^{\ast} \Vert_{JF_{X}^{\ast}}
\leq
\Vert\sum_{j=1}^{n}\lambda_{j}e_{j}^{\ast}\Vert_{X^{\ast}}\,\,.
$$
The proof is complete.
\end{proof}
\begin{remark}
If $Y$ denotes the subspace of $JF_{X}^{\ast}$ generated by the
biorthogonal functionals of the $Haar$ system, the previous
theorem yields that $JF_{X}^{*}/Y$ is isomorphic to
$X^{*}_{\Gamma}.$ In the particular case of James function space
$JF$, the corresponding quotient is isomorphic to
$\ell_{2}(\Gamma).$ This result is the analogue of the
corresponding result for the James tree space $JT$, \cite{ls}.
Next we will see that these results are no longer valid for the
class of the spaces $JF_{X}^{\ast}(\Omega)$, $\Omega$ open bounded
subset of $\mathbb{R}^{d_{0}}$, $d_{0}>1.$
\end{remark}
\begin{proposition}\label{p34}
Let $d_{0}>1$ and $\Omega=(0,1)^{d_{0}}.$ Let $(T_{j})_{j\in J}$
be a family in $\mathcal{P}(\Omega)$ such that there exists
$1\leq d\leq d_{0}$ with
$$
(0,1)\not\subset \{pr_{d}(x): x\,\,\mathrm{\,is\,a \,vertex\,of
}\, \,T_{j}\,\,\text{for some}\,\,j\in J\}\,.
$$
Then if $Y=\overline{\langle T_{j}^{*}:j\in J \rangle}$, the
quotient $JF_{X}^{*}(\Omega)/Y$ is not reflexive.
\end{proposition}
\begin{proof}
Assume that for $d=1$ there exists $\alpha\in (0,1)\setminus
\{pr_{1}(x): x$ is a vertex of $T_{j}$ for some $j\in J\}.$ We
choose $r_{n}\in (0,1)$ strictly decreasing to zero and set
$S_{n}=(\alpha,1)\times (0,r_{n})^{d_{0}-1}.$
\begin{claim} $ \{
\widehat{S_{n}^{\ast}} \}_{n}$ has no weakly converging
subsequences.
\end{claim}
\noindent {\it Proof of the Claim}. For $F$ a finite subset of
$J$ (i.e $F\in\mathcal{P}_{<\omega}(J)$) we set $Y_{F}=\langle
T^{*}_{j} : j\in F\rangle.$ Clearly for every $x^{*}\in
JF_{X}^{*}(\Omega)$,
$$
\Vert\widehat{x^{*}}\Vert=dist(x^{*},Y)=\inf\{dist(x^{*},Y_{F}):
F\in\mathcal{P}_{<\omega}(J)\}\,\,.
$$
To show that $\{\widehat{S_{n}^{*}}\}_{n}$ does not have
$w-$convergent subsequence it is enough to prove that for
$\varepsilon=\frac{1}{\Vert e_{1}+e_{2}\Vert_{X}}$ and every
$\{R_{k}^{*}\}_{k\in\mathbb{N}}$ convex block subsequence of
$\{S_{n}^{*}\}_{n}$ there exists $k_{1}<k_{2}$ such that
$$
\Vert \widehat{R_{k_{1}}^{*}}-\widehat{R_{k_{2}}^{*}}\Vert\geq
\varepsilon\,\,.
$$
To see this, we consider $G_{1}$,$G_{2}$ finite subsets of
$\mathbb{N}$ with $\max G_{1}=\ell< q=\min G_{2}$ and
$R_{1}=\sum_{n\in G_{1}}\alpha_{n}S_{n}^{*}$,
$R_{2}^{*}=\sum_{n\in G_{2}}\alpha_{n}S_{n}^{*}.$ Observe that
$$
\sum_{n\in G_{1}}\alpha_{n}\chi_{S_{n}}(w)=1\,\,
\mathrm{for\,all\,}\,w\in (\alpha,1)\times
(r_{q},r_{\ell})^{d_{0}-1}\,\,.
$$
Next we consider any $F\in\mathcal{P}_{<\omega}(J)$
and we show that
$$
dist(R_{1}^{*}-R_{2}^{*},Y_{F})\geq\frac{1}{ \Vert
e_{1}+e_{2}\Vert_{X}}\,,
$$
which immediately yields the claim.

Indeed, set
$$
\delta=\min\Bigl\{ \{\vert\alpha-pr_{1}(x)\vert:\,x\,\, \text{is a
vertex of}\,\, T_{j}, j\in F\}\cup
\{\alpha,1-\alpha\}\Bigl\}>0\,\,.
$$
Moreover we choose $r_{q}^{\prime}<r_{\ell}^{\prime}$ such that
$r_{q}<r_{q}^{\prime}<r_{\ell}^{\prime}<r_{\ell}$ and
\begin{center}
$(r_{q}^{\prime},r_{\ell}^{\prime})\cap\{pr_{d}(x) :x$ is a vertex
of $T_{j}, j\in F\}=\emptyset $ for all $d=2,3,\ldots, d_{0}.$
\end{center}
We set
$$
Q_{1}=(\alpha-\frac{\delta}{2},\alpha)\times
(r_{q}^{\prime},r_{\ell}^{\prime})^{d_{0}-1}\,\,\,\,
\text{and}\,\,\,\, Q_{2}=(\alpha,\alpha+\frac{\delta}{2})\times
(r_{q}^{\prime},r_{\ell}^{\prime})^{d_{0}-1}\,\,.
$$
Observe the following.
\begin{enumerate}
\item For $j\in F$ either $Q_{1}\cup Q_{2}\subset T_{j}$ or
$(Q_{1}\cup Q_{2})\cap T_{j}=\emptyset$.

\smallskip
\item $\sum_{n\in G_{1}}\alpha_{n}\chi_{S_{n}}(w)=0,\,
\forall w\in Q_{1}$ and $\sum_{n\in
G_{1}}\alpha_{n}\chi_{S_{n}}(w)=1,\,\forall w\in Q_{2}$.

\smallskip
\item
$\sum_{n\in G_{2}}\alpha_{n}\chi_{S_{n}}(w)=0,\,\forall w\in
Q_{1}\cup Q_{2}$.
\end{enumerate}
Consider now the element of $JF_{X}(\Omega)$ defined by
\begin{center}
$ f=\frac{\chi_{Q_{2}}-\chi_{Q_{1}}}{\mu(Q_{2})}$\,\,.
\end{center}
An easy computation yields that $ \Vert f\Vert=\Vert
e_{1}+e_{2}\Vert_{X}.$
Properties $(1),(2)$ and $(3)$, stated
above, imply that
$$ dist(R_{1}^{*}-R_{2}^{*}, Y_{F})\geq\langle
R_{1}^{*}-R_{2}^{*}, \frac{1}{\Vert e_{1}+e_{2}\Vert_{X}}f\rangle
= \langle R_{1}^{*},\frac{1}{\Vert e_{1}+e_{2}\Vert_{X}}f\rangle
=\frac{1}{\Vert e_{1}+e_{2}\Vert_{X}}\,\,.
$$
This completes the proof of the claim, and the proof of the
proposition.
\end{proof}
%
%
Let's pass now to some consequences of the above proposition.
\begin{proposition}\label{p35}
Let $\Omega$ be a bounded open subset of $\mathbb{R}^{d_{0}}$,
$d_{0}>1$, and $\{T_{j}\}_{j\in J}\subset\mathcal{P}(\Omega)$
such that for some $1\leq d\leq d_{0}$
$$
pr_{d}(\Omega)\not\subset\{pr_{d}(x) : x\,\,\text{is a vertex
of}\,\,T_{j}, j\in J\}.
$$
If $Y=\overline{\langle T_{j}^{*}: j\in J\rangle}$,  then
$JF_{X}^{*}(\Omega)/Y$ is not reflexive.
\end{proposition}
\begin{proof}
Since $pr_{d}(\Omega)$ is open, there exists
$S=\prod_{i=1}^{d_{0}}(\alpha_{i},\beta_{i})\subset\Omega$ such
that
\begin{center} $(\alpha_{d},\beta_{d})\not\subset\{pr_{d}(x):
x\,\,\text{is a vertex of}\,\, T_{j}, j\in J\}\,\,.$
\end{center}
The result is obtained with the same arguments as in the previous
proposition.
\end{proof}
%
%
\begin{theorem} \label{t36}
Let $\Omega$ be a bounded open subset of $\mathbb{R}^{d_{0}}$,
$d_{0}>1.$ Then for every separable  subspace $Y$ of
$JF_{X}^{*}(\Omega)$ the quotient $JF_{X}^{*}(\Omega)/Y$ is not
reflexive.
\end{theorem}
\begin{proof}
Since $\ell_{1}$ does not embed into $JF_{X}(\Omega)$, by Lemmas
\ref{l2}, \ref{l5} and Haydon's theorem \cite{Had}, we obtain that
$JF_{X}^{*}(\Omega)=\overline{\langle T^{*} :
T\in\mathcal{P}(\Omega)\rangle}.$ Hence for any $Y$ separable
subspace of $JF_{X}^{*}(\Omega)$ there exists a sequence
$\{T_{n}^{*}\}_{n}$ such that $Y\hookrightarrow
Z=\overline{\langle T_{n}^{*}:n\in\mathbb{N}\rangle}.$ Clearly
$\{T_{n}\}_{n}$ satisfies the assumption of Proposition~\ref{p35},
hence $JF_{X}^{*}(\Omega)/Z$ is not reflexive. This implies that
$JF_{X}^{\ast}(\Omega)/Y$ is also not reflexive.
\end{proof}
\begin{corollary}\label{c37}
Let $\Omega$ be a bounded open subset of $\mathbb{R}^{d_{0}}$,
$d_{0}>1.$ Then the space $JF_{X}^{\ast}(\Omega)$ is not
isomorphic to a quotient of $JF_{X}^{\ast}.$ In particular
$JF_{X}(\Omega)$ is not isomorphic to a subspace of $JF_{X}.$
\end{corollary}
\begin{proof}
On the contrary, assume that there exists a subspace $Z$ of
$JF^{\ast}_{X}$ such that $JF_{X}^{\ast}(\Omega)$ is isomorphic to
the quotient space $JF^{\ast}_{X}/Z.$ Let $\Delta$ be a countable
dense subset of $(0,1)$ and $Y=\overline{\langle\{I^{*}: I\,
\text{has endpoints in}\,\Delta\rangle\} }.$ Clearly the space $Y$
is separable. We set $W=\overline{\langle Y \cup Z\rangle}.$ It is
well known that $JF^{\ast}_{X}/W$ is isometric to
$JF^{\ast}_{X}/Y{\Bigl /} W/Y$ and also isometric to
$JF^{\ast}/Z{\Bigl /} W/Z.$ Moreover the quotient space $W/Z$ is
separable. Hence we have that
\begin{equation}\label{e33}
  JF_{X}^{\ast}/W\approx JF^{\ast}_{X}/Z{\Bigl /}
W/Z\approx JF^{\ast}_{X}(\Omega){\Bigl /}W/Z\,\,,
\end{equation}
and therefore by Theorem~\ref{t36}, $JF_{X}^{\ast}/W$ is not
reflexive. On the other hand, by Theorem~\ref{t32}, we have that
$JF_{X}^{\ast}/Y$ is reflexive, and therefore
\begin{equation}\label{e34}
JF_{X}^{\ast}/W\approx JF^{\ast}_{X} /Y{\Bigl/} W/Y\approx
X^{\ast}_{\Gamma}{\Bigl /}W/Y\,\,.
\end{equation}
>From (\ref{e33}) and (\ref{e34}) we derive a contradiction.

The second part follows from a duality argument.
\end{proof}
\section{The embedding of $JF_{X}$ into $JF_{X}(\Omega)$.}
\label{sembedding}
In this section we prove the following:
\begin{theorem}\label{t41}
Let $1\leq d_{0}\leq d_{1}$ and $\Omega$ be a bounded open subset
of $\mathbb{R}^{d_{1}}.$ Then $JF_{X}((0,1)^{d_{0}})$ is isometric
to a complemented subspace of $JF_{X}(\Omega).$
\end{theorem}
Since $JF_{X}((0,1)^{d_{1}})$ is isometric to $1-complemented$
subspace of $JF_{X}(\Omega)$, it is enough to prove the result for
$\Omega=(0,1)^{d_{1}}.$

We set $\mathcal{D}$ be the dense  subspace of
$L^{1}((0,1)^{d_{0}})$ consisting of the functions of the form
$x=\sum_{j=1}^{m}\alpha_{j}\chi_{R_{j}}$, where $\{R_{j}\}_{j}$
are disjoint parallelepipeds in $\mathcal{P}((0,1)^{d_{0}}).$

To each $x=\sum_{j=1}^{m}\alpha_{j}\chi_{R_{j}}\in\mathcal{D}$ we
correspond the vector $\tilde{x}
=\sum_{j=1}^{m}\alpha_{j}\chi_{R_{j}\times (0,1)^{d_{1}-d_{0}}}$
of $JF_{X}(\Omega).$ Denote by
$$
U:\mathcal{D}\mapsto JF_{X}(\Omega)
$$
the above assignment, which is a linear operator.
\begin{lemma}\label{l42}
Let $x\in\mathcal{D}$ and $x^{*}\in JF_{X}^{*}(\Omega)$ such that
$x^{*}=\sum_{i=1}^{n}b_{i}T_{i}^{*}$,
$\{T_{i}\}_{j=1}^{n}\subset\mathcal{P}(\Omega)$ disjoint, and
$\Vert\sum_{i=1}^{n}b_{i}e_{i}^{*}\Vert_{X^{*}}\leq 1.$ Then
\begin{enumerate}
\item
$\,\, \Vert x\Vert_{JF_{X}((0,1)^{d_{0}})} \leq
\Vert\tilde{x}\Vert_{JF_{X}(\Omega)}$.
\item $\,\, x^{*}(\tilde{x}) \leq \Vert x
\Vert_{ JF_{X}((0,1)^{d_{0}})}$.
\end{enumerate}
\end{lemma}
\begin{proof}
$(1).$ This follows easily. Indeed  for every
$\{S_{l}\}_{l=1}^{k}\subset\mathcal{P}((0,1)^{d_{0}})$ disjoint,
we consider the disjoint family $\{S_{l}\times
(0,1)^{d_{1}-d_{0}}\}_{l=1}^{k}$ of $(0,1)^{d_{1}}$, and notice
that
$$
\Vert \sum_{l}
\left(\int_{S_{l}}\sum_{j=1}^{m}\alpha_{j}\chi_{R_{j}}\right)
e_{l}\Vert = \Vert \sum_{l} \left(\int_{S_{l}\times
(0,1)^{d_{1}-d_{0}}} \sum_{j=1}^{m}\alpha_{j}\chi_{R_{j}\times
(0,1)^{d_{1}-d_{0}}}\right)e_{l}\Vert\,\,.
$$
Taking the supremum in both sides we obtain $(1).$

$(2)$ To see the second inequality, assume that
$x=\sum_{j=1}^{m}\alpha_{j}\chi_{R_{j}}$,
$x^{*}=\sum_{i=1}^{n}b_{i}T_{i}^{*}$ are given. Denote by
$\pi_{1}:\mathbb{R}^{d_{1}}\mapsto\mathbb{R}^{d_{0}}$,
$\pi_{2}:\mathbb{R}^{d_{1}}\mapsto\mathbb{R}^{d_{1}-d_{0}}$ the
natural projections of $\mathbb{R}^{d_{1}}$ onto the two
orthogonal subspaces $\mathbb{R}^{d_{0}}$,
$\mathbb{R}^{d_{1}-d_{0}}.$

Assume additionally, that the families $\{R_{j}\}_{j=1}^{m}$,
$\{T_{i}\}_{i=1}^{n}$ satisfy the following property
\begin{equation}\label{e41}
\text{For
every}\,\,j=1,\ldots,m\,,\,i=1,\ldots,n\,\,\text{either}\,\,
R_{j}\subset\pi_{1}(T_{i})\,\,\text{or}\,\,
R_{j}\cap\pi_{1}(T_{i})=\emptyset\,\,.
\end{equation}
If (\ref{e41}) fails, then we rewrite $x$ as
$\sum_{\ell=1}^{k}\alpha_{\ell}\chi_{R_{\ell}^{\prime}}$ so that
the families $\{R_{\ell}^{\prime}\}_{\ell}$, $\{T_{i}\}_{i}$
satisfying (\ref{e41}).

\noindent  Next we choose family $\{ Q_{\ell} \}_{ \ell
=1}^{k}\subset\mathcal{P}((0,1)^{d_{1}-d_{0}})$ such that
\begin{enumerate}
\item $\{ Q_{\ell}\}_{\ell=1}^{k}$ are pairwise disjoint
parallelepipeds.
\item For every $\ell=1,\ldots,k$, $i=1,\ldots,n$ either
$Q_{\ell}\subset\pi_{2}(T_{i})$ or
$Q_{\ell}\cap\pi_{2}(T_{i})=\emptyset.$
\item If $B_{i}=\{\ell: Q_{\ell}\subset\pi_{2}(T_{i})\}$, then
$T_{i}=\cup_{\ell\in B_{i}}(\pi_{1}(T_{i})\times Q_{\ell})$
almost everywhere.
\end{enumerate}
We also set
\begin{center}
$A_{j}=\{i: R_{j}\subset \pi_{1}(T_{i})\}$\,\,\,\,\,\, for
$j=1,\ldots,m$
\end{center}
The above properties yield that
\begin{align*}
x^{*}(\tilde{x})=& \sum_{i=1}^{n}b_{i}T_{i}^{*}(\tilde{x})
=\sum_{j=1}^{m}\alpha_{j}\mu_{d_{0}}(R_{j}) \sum_{i\in
A_{j}}b_{i}\sum_{\ell\in B_{i}}\mu_{d_{1}-d_{0}}(Q_{\ell})
\\
=& \sum_{\ell=1}^{k}\mu_{d_{1}-d_{0}}(Q_{\ell})\sum_{i:\ell\in
B_{i}}b_{i}\sum_{j: i\in
A_{j}}\alpha_{j}\mu_{d_{0}}(R_{j})\\
= &\sum_{\ell=1}^{k}\mu_{d_{1}-d_{0}}(Q_{\ell})\sum_{i:\ell\in
B_{i}}b_{i}(\pi_{1}(T_{i}))^{*}
(\sum_{j: i\in A_{j}}\alpha_{j}\chi_{R_{j}})\\
= &\sum_{\ell=1}^{k}\mu_{d_{1}-d_{0}}(Q_{\ell})\sum_{i:\ell\in
B_{i}}b_{i}(\pi_{1}(T_{i}))^{*}(x)\,\,.
\end{align*}
Recall that $\mu_{d_{0}}$, $\mu_{d_{1}-d_{0}}$ denotes the
Lebesgue measure on $\mathbb{R}^{d_{0}}$,
$\mathbb{R}^{d_{1}-d_{0}}$ respectively. To see that
$\sum_{i}b_{i}T_{i}^{*}(\tilde{x})\leq \Vert x\Vert_{
JF_{X}((0,1)^{d_{0}})}$ we notice that property $(1)$ yields that
\begin{equation}\label{e42}
\sum_{\ell}\mu_{d_{1}-d_{0}}(Q_{\ell})\leq 1\,\,.
\end{equation}
For fixed $\ell$, if $1\leq i_{1}\not= i_{2}\leq n$ are such that
$\ell\in B_{i_{1}}$, $\ell\in B_{i_{2}}$, it holds that
$\pi_{1}(T_{i_{1}})\cap\pi_{1}(T_{i_{2}})=\emptyset.$ Therefore
setting
$$
y^{*}_{\ell}=\sum_{i:\ell\in B_{i}}b_{i}\pi_{1}(T_{i})^{*}
\quad\text{for}\,\, \ell=1,\ldots,k\,\,,
$$
we conclude that $y^{*}_{\ell}\in B_{JF_{X}^{*}}.$

Finally  set
$$
y^{*}=\sum_{\ell=1}^{k}\mu_{d_{0}-1}(Q_{\ell})y_{\ell}^{*}\,.
$$
Observe that (\ref{e42})  yields that $y^{*}$ is a subconvex
combination of $\{y^{*}_{\ell}\}_{\ell=1}^{k}$, hence $y^{*}\in
B_{JF_{X}^{*}((0,1)^{d_{0}})}$ and $y^{*}(x)=x^{*}(\tilde{x}).$
This completes the proof of the lemma.
\end{proof}
%
\begin{proof}[Proof of Theorem~\ref{t41}] From Lemma~\ref{l42}
it follows that $U:\mathcal{D}\mapsto JF_{X}(\Omega)$ is an
isometry, which is extended to an isometry of
$JF_{X}((0,1)^{d_{0}})$ into $JF_{X}(\Omega).$ It remains to show
that $U(JF_{X}(0,1)^{d_{0}})$ is $1-complemented$ subspace.
Indeed, we set
$$
Y=\overline{\langle R\times(0,1)^{d_{1}-d_{0}}:
R\in\mathcal{P}((0,1)^{d_{0}})\rangle}\subset
JF_{X}^{*}(\Omega)\,\,,
$$
and $Q:JF_{X}(\Omega)\mapsto JF_{X}(\Omega)/Y_{\bot}.$ Since $U$
is an isometry, we obtain that  $B_{Y}$ $1-norms$ the subspace
$U(JF_{X}(0,1)^{d_{0}})$, hence $Q\circ U$ is also an isometry. To
see that is onto, we observe that for every
$T\in\mathcal{P}(\Omega)$ there exists $0\leq\lambda\leq 1$, such
that $\chi_{T}-\lambda\chi_{\pi_{1}(T)\times
(0,1)^{d_{1}-d_{0}}}\in Y_{\bot}.$
\newline
This completes the proof of the theorem.
\end{proof}
\section{Subsequences of Rademacher functions
equivalent to  $c_{0}$ basis.}\label{sc0}
Before stating the next
definition we introduce some notation. Let $(n_{k})_{k}$ be an
increasing sequence of $\mathbb{N}$ and $(\sigma_{k})_{k}$ a
sequence of successive subsets of $\N$ with $\#\si_{k}=
2^{n_{k}}$. We denote by
$\lambda_{k}^{-1}=\Vert\sum_{n\in\si_{k}}e_{n}\Vert_{X}$, and we
set $u_{k}=\lambda_{k}\sum_{n\in\si_{k}}e_{n}$, which clearly
satisfies $\Vert u_{k}\Vert_{X}=1.$
\begin{definition}\label{d51}
Let $X$ be a reflexive Banach space with $1-$symmetric basis
$(e_{n})_{n}$. The space $X$ satisfies the {\it Convex Combination
Property} ($CCP$) if there exist a strictly increasing sequence
$(n_{k})_{k}$ and $C>0$ such that the following is fulfilled:

For $(\si_{k})_{k}$, $(u_{k})_{k}$ as above with
$\#\si_{k}=2^{n_{k}}$, every $(I_{k})_{k}$ with $I_{k}\subset
\sigma_{k}$ and $\sum_{k}\frac{\# I_{k}}{\#\si_{k}}\leq 1$ we have
that
$$
\Vert\sum_{k=1}^{\infty}\lambda_{k}\sum_{n\in I_{k}}
e_{n}\Vert_{X}\leq C\,\,.
$$
\end{definition}
Our goal is to prove the following:
\begin{theorem} \label{500}The following
are equivalent:
\begin{enumerate}
\item The space $X$ satisfies $CCP$.
\item The normalized sequence
$(\frac{r_{n}}{\Vert r_{n}\Vert})_{n\in\N}$ in $JF_{X}$ of
Rademacher  functions contains a subsequence equivalent to the
usual basis of $c_{0}$.
\end{enumerate}
\end{theorem}

\begin{remark} It is an easy exercise that $\ell_{p}$,
$1<p<\infty,$ have $CCP$. Lorentz space $d(w_{1},p)$,
$1<p<\infty$, $w_{1}=(\frac{1}{n})_{n}$ fails $CCP$. Indeed,
choose a rapidly increasing sequence $(n_{k})_{k}$ of integers,
$(\sigma_{k})$ subsets of $\N$ with $\#\sigma_{k}=2^{n_{k}}$ and
$I_{k}\subset\sigma_{k}$ with $\# I_{k}=
\frac{\#\sigma_{k}}{2^{k}}$. From the definition of the norm of
the space $d(w_{1},p)$, using that $\ln(n)\approx
\sum_{i=1}^{n}\frac{1}{i}$, we easily see that
$$
\Vert\sum_{n\in\sigma_{k}}e_{n}\Vert_{d(w_{1},p)} \approx
(\ln(2^{n_{k}}))^{1/p}\,\,,
$$
and
$$
\Vert\sum_{k}\sum_{n\in I_{k}} \frac{e_{n}}{
\Vert\sum_{i\in\sigma_{k}}e_{i}\Vert}\Vert_{d(w_{1},p)} \gtrapprox
\left(\sum_{k}\frac{\ln(2^{n_{k}-k})- \ln(2^{n_{k-1}-k+2})}{
\ln(2^{n_{k}})}\right)^{1/p}\to\infty\,\,.
$$
We do not know if $JF_{d(w_{1},p)}$ contains $c_{0}$.
\end{remark}
The next proposition yields that $CCP$ implies a formally stronger
property.
\begin{proposition}\label{p52} Let $X$ have the $CCP$ and
$(n_{k})_{k}$, $(\si_{k})_{k}$, $(u_{k})_{k}$,
$(\lambda_{k})_{k}$, $C>0$ as before. Then for every sequence
$(\alpha_{n})_{n\in\sigma}$, $\si=\cup_{k}\si_{k}$ such that
\begin{align}
0\leq \alpha_{n}\leq 1\,\,\,\,\,\,\text{and}\,\,\,\,\,\,
\sum_{k=1}^{\infty}
\frac{\sum_{n\in\si_{k}}\alpha_{n}}{\#\si_{k}}\leq
1\,\,, \label{e51}\\
\intertext{we have that} \Vert\sum_{k=1}^{\infty}
\sum_{n\in\si_{k}}\lambda_{k}\alpha_{n}e_{n}\Vert_{X}\leq
C+1\,\,.\label{e52}
\end{align}
\end{proposition}
\begin{proof}
It is enough to show that for every $k_{0}\in\N$ and every
$(\alpha_{n})_{n\in\cup_{k=1}^{k_{0}}\si_{k}}$ satisfying
(\ref{e51}) we have that
$$
\Vert\sum_{k=1}^{k_{0}}\sum_{n\in\si_{k}}
\lambda_{k}\alpha_{n}e_{n}\Vert_{X}\leq C+1\,\,.
$$
Fix $k_{0}$ and set
$$
K= \Bigl\{\sum_{k=1}^{k_{0}}
\sum_{n\in\si_{k}}\lambda_{k}\alpha_{n}e_{n}
:\,\,(\alpha_{n})_{n\in\cup_{1}^{k_{0}}\si_{k}}\,\,
\text{satisfies}\,(\ref{e51})\Bigl\}\,\,.
$$
\noindent Then $K$ is a closed convex bounded subset of a finite
dimensional subspace of $X$, hence it is convex and compact. It is
easy to see that
\begin{center}
If $(\alpha_{n})_{n\in\cup_{k=1}^{k_{0}}\si_{k}}$ satisfies
(\ref{e51}) and there exists \,$n_{1}$,
$n_{2}\in\cup_{k=1}^{k_{0}}\si_{k}$ with $n_{1}\not=n_{2}$\, and\,
$0<\alpha_{n_{1}},\alpha_{n_{2}}<1$ then\,
$\sum_{k=1}^{k_{0}}\sum_{n\in\si_{k}}\lambda_{k}\alpha_{n}e_{n}$
is not an extreme point of $K$.
\end{center}
Hence every $x\in Ex(K)$ is of the form
$x=\sum_{k=1}^{k_{0}}\sum_{n\in
I_{k}}\lambda_{k}e_{n}+\alpha_{n}\lambda_{k}e_{n}$, and from $CCP$
we obtain that $\Vert x\Vert_{X}\leq C+1$. This yields that for
every $x\in K$, $\Vert x\Vert_{X}\leq C+1$.
\end{proof}
Next we summarize some simple properties of Rademacher functions
$(r_{n})_{n}$.
\begin{lemma}\label{l53} The following hold
\begin{enumerate}
\item For every interval $I\subset [0,1]$, $n\in\N$,
$\vert\int_{I}r_{n}d\mu\vert\leq\frac{1}{2^{n}}$.
\item If
$\lambda_{k}^{-1}=\Vert\sum_{i=1}^{2^{n}}e_{i}\Vert_{X}$ then
$\Vert r_{n}\Vert_{JF_{X}}=\lambda_{n}^{-1}2^{-n}$.
\item Denote $\tilde{r}_{n}=\frac{r_{n}}{\Vert
r_{n}\Vert_{JF_{X}}}=\lambda_{n}2^{n}r_{n}$. Then
$\int_{\frac{k}{2^{n}}}^{\frac{k+1}{2^{n}}}
\tilde{r}_{n}=\pm\lambda_{n}$.
\item The reflexivity of $X$ implies that
$\lambda_{n}^{-1}\to\infty$, $\lambda_{n}^{-1}2^{-n}\to 0$.
\item
For each interval $I$, $\vert\int_{I}\tilde{r}_{n}\vert\leq
\lambda_{n}{}_{\stackrel{\longrightarrow}{n}}0$, hence
$(\tilde{r}_{n})_{n}$ is weakly null sequence in $JF_{X}$.
\end{enumerate}
\end{lemma}
\begin{proof} (1) is well known, (2) follows from
Proposition~\ref{p21}, (3) It is easy. (4) It is well known,
\cite{LT}, \cite{pe}. (5) follows from (1) and (4).
\end{proof}
\begin{notation}{\rm (a) In the Section~\ref{sHaar}, for $f\in
L^{1}(\mu)$ and $Q=\{I_{j}\}_{j=1}^{n}$ partition of $(0,1)$, the
quantity $\tau(Q,f)$ was defined. This is extended to each $f\in
JF_{X}$ as follows:
$$
\tau(Q,f)=\Vert\sum_{j=1}^{n}I^{*}_{j}(f)e_{j}\Vert_{X}\,\,.
$$
(b). We recall that for $Q=\{I_{j}\}_{j=1}^{n}$ partition of
$(0,1)$, the width of $Q$ is defined as
$\delta(Q)=\max\{\mu(I_{j}): j=1,\dots,n\}$}\,\,.
\end{notation}
\begin{lemma}\label{l54} Let $f\in JF_{X}(\Omega)$. Then
$$
\lim_{\e\to 0}\sup\{\tau(Q,f): \delta(Q)\leq\e\}=0\,\,.
$$
\end{lemma}
\begin{proof} Notice that if $X$ has a symmetric basis
$(e_{n})_{n}$ and it does not contain $\ell_{1}$, then for every
$\e>0$ there exists $\delta>0$ such that for every
$(\alpha_{n})_{n}$, $0\leq\alpha_{n}\leq 1$, $\max\{\alpha_{n}:
n\in\N\}\leq\delta$ and $\sum_{n}\alpha_{n}\leq 1$ we have that
$\Vert\sum_{n=1}^{\infty}\alpha_{n}e_{n}\Vert_{X}<\e$,
Lemma~\ref{pe1}. This property together with the uniform
integrability of $f\in L^{1}(\mu)$ yields the conclusion for $f\in
L^{1}(\mu)$. The general case follows easily from the density of
$L^{1}(\mu)$ in $JF_{X}$.
\end{proof}
\begin{proposition}\label{p55} Assume that $(\tilde{r}_{n})_{n\in M}$ is a
subsequence of the normalized Rademacher functions which endowed
with $\Vert\cdot\Vert_{JF_{X}}$ is equivalent to the unit vector
basis of $c_{0}$. Then $X$ satisfies the $CCP$.
\end{proposition}
\begin{proof} Choose $C>0$ such that for any finite subset $F$ of
$M$, $\Vert\sum_{n\in F}\tilde{r}_{n}\Vert_{JF_{X}}\leq C$. Next
choose inductively $n_{1}<\ldots<n_{k}<\ldots$, such that
$\{n_{k}\}\subset M$ and
\begin{equation}\label{e53}
\sup\left\{\tau(\tilde{r}_{n_{k}},Q) :
\delta(Q)\leq\frac{1}{2^{k+1}}\right\}\leq\frac{1}{2^{k}}\,\,.
\end{equation}
\noindent We claim that $\{n_{k}\}_{k}$, $C+1$ satisfy $CCP$ in
$X$.

Indeed given $\Lambda_{1}\subset \si_{1},\ldots,\Lambda_{k}\subset
\si_{k}$ such that
$\sum_{j=1}^{k}\frac{\#\Lambda_{j}}{\#\si_{j}}\leq 1$, setting
$d_{j}=\# L_{j}$, we choose a partition
$\{I^{1}_{1},\ldots,I^{1}_{d_{1}},\ldots$,
$I^{k}_{1},\ldots,I^{k}_{d_{k}}\}$ of $(0,1)$ as follows:

Let $A_{1}=[0,\frac{\# \Lambda_{1}}{2^{n_{1}}}]$ and
$A_{j}=[\sum_{i<j}\frac{\# \Lambda_{i}}{2^{n_{i}}}, \sum_{i\leq j}
\frac{\# \Lambda_{i}}{2^{n_{i}}}]$ for $j=2,\ldots,k$.
For every $j\leq k$ we consider $d_{j}$ intervals
$\{I_{i}^{j}\}_{i=1}^{d_{j}}$ which form a partition of $A_{j}$,
and each $I_{i}^{j}$ is of the form
$I_{i}^{j}=[\frac{\ell}{2^{n_{j}}},\frac{\ell+1}{2^{n_{j}}}]$, for
an appropriate $\ell$. Observe that
$\mu(\cup_{i=1}^{d_{j}}I^{j}_{i})=\frac{\#\Lambda_{j}}{\#\si_{j}}$.
Further for every $1\leq j\leq k$, $i=1,\ldots,d_{j}$
\begin{align}
&\vert\int_{I^{j}_{i}}\tilde{r}_{n_{j}}\vert=\lambda_{j}=
\Vert\sum_{n\in\si_{j}}e_{n}\Vert_{X}^{-1}\,\,.\notag\\
\intertext{Therefore}
\Vert\sum_{j=1}^{k}\sum_{n\in\Lambda_{j}}&\lambda_{j}e_{n}
\Vert_{X}= \Vert\sum_{j=1}^{k}\sum_{i=1}^{d_{j}}\left(
\int_{I^{j}_{i}}\tilde{r}_{n_{j}}d\mu\right)e_{n_{i}}\Vert_{X}\,\,.
\label{e54}
\end{align}
Further
\begin{align}
C\geq&\Vert\sum_{j=1}^{k}\tilde{r}_{n_{j}}\Vert_{JF_{X}}\geq
\Vert\sum_{j=1}^{k}\sum_{i=1}^{d_{j}}\left(\int_{I^{j}_{i}}
\sum_{l=1}^{k}\tilde{r}_{n_{l}}d\mu\right)e_{n_{i}}\Vert_{X}
\notag\\
= &\Vert\sum_{j=1}^{k}\sum_{i=1}^{d_{j}}\left(\int_{I^{j}_{i}}
\sum_{l\leq j}\tilde{r}_{n_{l}}d\mu\right)e_{n_{i}}\Vert_{X}
\notag\\
\geq &\Vert\sum_{j=1}^{k}\sum_{i=1}^{d_{j}}\left(\int_{I^{j}_{i}}
\tilde{r}_{n_{j}}d\mu\right)e_{n_{i}}\Vert_{X}-
\sum_{j=1}^{k}\tau(\tilde{r}_{n_{j}}, Q_{j})\,\,,\label{e55}
\end{align}
where
$Q_{j}=\cup_{l>j}\left\{\{I^{l}_{i}\}_{i=1}^{d_{l}}\right\}$.
Condition (\ref{e53}) yields that
$$
\sum_{j=1}^{k}\tau(\tilde{r}_{n_{j}},Q_{j})\leq 1.
$$
Setting together (\ref{e54}), (\ref{e55}) we obtain the desired
result.
\end{proof}
\begin{proposition}\label{p36} Assume that $X$ satisfies $CCP$.
Then there exists a subsequence of Rademacher functions
equivalent, in $JF_{X}$ norm,
to the unit vector basis of $c_{0}$.
\end{proposition}
\begin{proof} Let $(n_{k})_{k}$, $C>0$ witness the presence of
$CCP$ in $X$. We inductively choose a subsequence
$(n_{k_{\ell}})_{\ell}$ of $(n_{k})_{k}$ and a decreasing sequence
$(\e_{\ell})_{\ell}\subset (0,1)$ satisfying the following
properties:
\begin{enumerate}
\item  For each $\ell\in\N$,
$$
\sup\left\{\tau(\tilde{r}_{n_{k_{\ell}}},Q):\,\,\delta(Q)\leq
\e_{\ell}\right\}<\frac{1}{4^{\ell}}\,\,.
$$
\item For each $\ell>1$,
$$
\sup\{\tau(\tilde{r}_{n_{k_{\ell}}},Q) :\,\,\min\{\mu(I): I\in
Q\}\geq\e_{\ell-1}\}<\frac{1}{4^{\ell}}\,\,.
$$
\end{enumerate}
The inductive choice proceeds as follows. We set $k_{1}=1$ and
from Lemma~\ref{l54} there exists $\e_{1}>0$ such that for
$\ell=1$ the inductive assumption $(1)$ is fulfilled. Observe that
for every partition $Q$ of $(0,1)$ such that $\min\{\mu(Q): I\in
Q\}\geq\e_{1}$ satisfies $\#Q<\frac{1}{\e_{1}}$. Hence
Lemma~\ref{l53}(1) yields that there exists $k_{2}$ such that the
inductive assumption $(2)$ is fulfilled. The general inductive
step follows the same argument.
\begin{claim} The sequence $(\tilde{r}_{n_{k_{\ell}}})_{\ell}$ is
equivalent to the unit vector basis of $c_{0}$.
\end{claim}
For this, we show that the inductive assumptions $(1)$ and $(2)$
together with $CCP$ in $X$ yield that for every $d\in\N,$
\begin{equation}\label{e56}
\Vert\sum_{\ell=1}^{d}\tilde{r}_{n_{k_{\ell}}}\Vert_{JF_{X}}\leq
3(C+1)+1\,\,.
\end{equation}
Since every subsequence of $(\tilde{r}_{n_{k_{\ell}}})_{\ell}$
satisfies $(1)$ and $(2)$, we obtain that every subsequence of
$(\tilde{r}_{n_{k_{\ell}}})_{\ell}$ also satisfies (\ref{e56}),
and this will end the proof.

To see (\ref{e56}), we consider $Q=\{I_{j}\}_{j=1}^{q}$ arbitrary
partition of $(0,1)$ and show that
\begin{equation}\label{e57}
\tau(\sum_{\ell=1}^{d}\tilde{r}_{n_{k_{\ell}}},Q)\leq
3(C+1)+1\,\,.
\end{equation}
Consider the partition of $Q$ into $\{Q_{\ell}\}_{\ell=1}^{d}$
where $Q_{\ell}=\{I_{j}: \e_{\ell}\leq\mu(I_{j})<\e_{\ell-1}\}$,
where $\e_{0}=0$. Observe that $(1)$ and $(2)$ implies that for
$\ell=1,\ldots,d$
\begin{equation}\label{e58}
\tau(\tilde{r}_{n_{k_{\ell}}},\cup_{\ell^{\prime}\not=\ell}
Q_{\ell^{\prime}})<\frac{2}{4^{\ell}}\,\,.
\end{equation}
\noindent Hence
\begin{align*}
\tau(\sum_{\ell=1}^{d}\tilde{r}_{n_{k_{\ell}}},Q)&\leq
\tau(\sum_{\ell=1}^{d}(\tilde{r}_{n_{k_{\ell}}}{}_{|\cup
Q_{\ell}}),Q) +\sum_{\ell=1}^{d}\tau(\tilde{r}_{n_{k_{\ell}}},
\cup_{\ell^{\prime}\not=\ell}Q_{\ell^{\prime}})\\
&\leq
\tau(\sum_{\ell=1}^{d}(\tilde{r}_{n_{k_{\ell}}}{}_{|\cup
Q_{\ell}}),Q)+1\,\,.
\end{align*}
Here $\tilde{r}_{n_{k_{\ell}}}{}_{|\cup Q_{\ell}}$ denotes the
restriction of $\tilde{r}_{n_{k_{\ell}}}$ on the set $\cup\{I:
I\in Q_{\ell}\}$. In the last step we show that
$$
\tau(\sum_{\ell=1}^{d}(\tilde{r}_{n_{k_{\ell}}}{}_{|\cup
Q_{\ell}}), Q)\leq 3(C+1)\,\,.
$$
For this we split each $Q_{\ell}$ into three set $Q_{\ell}^{1}$,
$Q^{2}_{\ell}$, $Q^{3}_{\ell}$ as follows:
\begin{align*}
Q^{1}_{\ell}&=\{ I\subset Q_{\ell} :\,\,\exists\,\, 0\leq m <
2^{n_{k_{\ell}}},\,
(\frac{m}{2^{n_{k_{\ell}}}},\frac{m+1}{2^{n_{k_{\ell}}}})
\subset I \}\,\,, \\
Q^{2}_{\ell}&=\{ I\subset Q_{\ell} :\,\,\exists\,\, 0\leq m <
2^{n_{k_{\ell}}},\, I\subset(\frac{m}{2^{n_{k_{\ell}}}},
\frac{m+1}{2^{n_{k_{\ell}}}})\}\,\,,\\
Q^{3}_{\ell}&=\{ I\subset Q_{\ell} :\,\,\exists\,\, 0\leq m <
2^{n_{k_{\ell}}},\, \frac{m}{2^{n_{k_{\ell}}}}<\min
I<\frac{m+1}{2^{n_{k_{\ell}}}} <\max
I<\frac{m+2}{2^{n_{k_{\ell}}}} \}\,\,.
\end{align*}
Clearly $Q^{1}_{\ell}\cup Q^{2}_{\ell}\cup Q^{3}_{\ell}=Q_{\ell}$,
and we set $Q^{i}=\cup_{\ell=1}^{d}Q^{i}_{\ell}$ for $i=1,2,3$.
With the aid of $CCP$ we show that
$$
\tau(\sum_{\ell=1}^{d}\tilde{r}_{n_{k_{\ell}}}{}_{|\cup Q_{\ell}},
Q^{i})\leq C+1\,\,,
$$
which yields the entire proof.

\noindent We prove it for $i=2$, which is the most complicated
case. The other two cases follow from similar arguments.

For this, we choose $\si_{1}^{\prime}<\ldots<\si_{d}^{\prime}$
successive subsets of $\N$ with $\#\si^{\prime}_{\ell}=\#
Q_{\ell}^{2}$. Then
$$
\tau(\sum_{\ell=1}^{d}\tilde{r}_{n_{k_{\ell}}}{}_{|\cup Q_{\ell}},
Q^{2})= \Vert\sum_{\ell=1}^{d} \sum_{m\in\si_{\ell}^{\prime}}
\left(
\int_{I_{m}}\tilde{r}_{n_{k_{\ell}}}d\mu\right)e_{m}\Vert_{X}\,\,.
$$
We decompose each $Q_{\ell}^{2}$ into
$\{Q_{\ell,s}\}_{s=0}^{2^{n_{k_{\ell}}}}$ where $Q_{\ell,s}=\{
I\in Q_{\ell}^{2} :\, I\subset
(\frac{s}{2^{n_{k_{\ell}}}},\frac{s+1}{2^{n_{k_{\ell}}}})\}$.
Observe that for $I\in Q_{\ell}^{2}$
$$
\vert\int_{I}\tilde{r}_{n_{k_{\ell}}}d\mu\vert=
\mu(I)\cdot\lambda_{k_{\ell}}\cdot 2^{n_{k_{\ell}}}\,\,,
$$
where $\lambda_{k_{\ell}}=
\Vert\sum_{i=1}^{2^{n_{k_{\ell}}}}e_{i}\Vert_{X}^{-1}$. Hence
$$
\sum_{I\in Q_{\ell,s}}
\vert\int_{I}\tilde{r}_{n_{k_{\ell}}}d\mu\vert=\mu(\cup_{I\in
Q_{\ell,s}}I)\lambda_{k_{\ell}}2^{n_{k_{\ell}}}\,\,.
$$
Lemma~\ref{l22} yields that
$$
\Vert\sum_{\ell=1}^{d}\sum_{m\in\si_{\ell}^{\prime}}\left(
\int_{I_{m}}\tilde{r}_{n_{k_{\ell}}}d\mu\right)e_{m}\Vert_{X} \leq
\Vert \sum_{\ell=1}^{d}\sum_{m\in\si_{\ell}}\mu(\cup_{I\in
Q_{\ell,s_{m}}}I)\lambda_{k_{\ell}}2^{n_{k_{\ell}}}e_{m}\Vert_{X}\,\,,
$$
\noindent where $\si_{\ell}\subset\si_{\ell}^{\prime}$ with
$\#\si_{\ell}\leq 2^{n_{k_{\ell}}}$. Here $(\ell,s_{m})$ denotes a
one to one corresponding of $\si_{\ell}$ onto
$\{\ell,s\}_{s=0}^{2^{n_{k_{\ell}}}-1}.$ Also
$\sum_{\ell=1}^{d}\sum_{m\in\si_{\ell}}\mu(\cup_{I\in
Q_{\ell,s_{m}}}I)\lambda_{k_{\ell}}2^{n_{k_{\ell}}}e_{m}$
satisfies the assumptions of Proposition~\ref{p52}, hence
$$
\Vert \sum_{\ell=1}^{d}\sum_{m\in\si_{\ell}}\mu(\cup_{I\in
Q_{\ell,s_{m}}}I)\lambda_{k_{\ell}}2^{n_{k_{\ell}}}e_{m}\Vert_{X}
\leq C+1\,\,.
$$
\end{proof}
\begin{proof}[Proof of Theorem \ref{500}] It follows from
Proposition \ref{p55} and \ref{p36}.
\end{proof}

\section{The space $V_{X}$ of functions of
bounded $X-$variation.}\label{sVX} In the final section we present
a representation of $JF_{X}$ and $JF_{X}^{**}$ as function spaces
of bounded $X$-variation, which generalizes the representation of
$JF_{p}$ as spaces of functions of bounded $p-$variation pointed
out by J.Lindendstaruss and C.Stegall \cite{ls}, and used also by
S.V.Kisliakov \cite{K}, in his proof that $\ell_{1}$ does not
embed in $JF_{p}$.
\newline
Let $f:\; [0,1]\to\mathbb{R}$. We adopt the following notation.

For $\mathcal{P}=\{t_{i}\}_{i=0}^{n}$ a partition of $[0,1]$ and
$X$  a reflexive Banach space with 1-symmetric basis we set
$$
\alpha_{X}(f,\mathcal{P})=\Vert\sum_{i=0}^{n-1}(f(t_{i+1})
-f(t_{i}))e_{i}\Vert_{X}\,\,,
$$
and
$$
\Vert f\Vert_{V_{X}}=\sup\{\alpha_{X}(f,\mathcal{P})
:\mathcal{P}\,\,\text{partition of}\,\,[0,1]\}\,\,.
$$
We also set
$$
V_{X}=\{ f: f(0)=0\,\,\text{and}\,\,\Vert
f\Vert_{V_{X}}<\infty\}\,\,,
$$
and
$$
V_{X}^{0}=\{f\in V_{X}:\lim_{\varepsilon\to
0}\sup\{\alpha_{X}(f,\mathcal{P}):\,\delta(\mathcal{P})\leq\e\}=0\}\,\,.
$$
It is easy to see that $V_{X}$, $V_{X}^{0}$ are Banach spaces
endowed with $\Vert\cdot\Vert_{V_{X}}$. Further Lemma \ref{e54}
yields that the Voltera operator $ V(f)(t)=\int_{0}^{t}f(x)d\mu$
defines an isometry from $JF_{X}$ into $V_{X}^{0}$, which is
actually onto. Let's observe that the dual isometry maps $I^{*}$,
where $I=[t_{1},t_{2}]$, to $\delta_{t_{2}}-\delta_{t_{1}}$. Hence
the set
$$
\mathcal{S}=\left\{\sum_{n=1}^{\infty}a_{n}I_{n}^{*} :
\{I_{n}\}_{n=1}^{\infty}\,\,\,\text{pairwise
disjoint}\,\,\,\text{and}
\,\,\,\Vert\sum_{n=1}^{\infty}a_{n}e_{n}^{*} \Vert_{X^{*}}\leq
1\right\}
$$ is mapped onto the set
\begin{equation}\label{e61}
\mathcal{K}=\left\{\sum_{n=1}^{\infty}a_{n}(\delta_{d_{n}}-
\delta_{c_{n}}) :\{(c_{n},d_{n})\}_{n}\,\,\text{pairwise disjoint
and}\,\,\Vert\sum_{n=1}^{\infty}a_{n}e_{n}^{*} \Vert_{X^{*}}\leq
1\right\}\,,
\end{equation}
and so the set $\mathcal{K}$ shares the properties proved for the
set $\mathcal{S}$. Namely
\begin{enumerate}
\item $\mathcal{K}$ is $w^{*}-$compact and norming.
\item
$B_{(V_{X}^{0})^{*}}=
\overline{\rm{co}}^{\Vert\cdot\Vert}(\mathcal{K})$.
\newline
\noindent These two properties yield that
\item The space $V_{X}=(V_{X}^{0})^{**}$.
\item A bounded sequence $(f_{n})_{n}\subset V_{X}^{0}$ is
$w-$Cauchy iff $(f_{n}(t))_{n}$ is convergent for all $t\in[0,1]$.
\end{enumerate}
The next theorem summarize the above observations.
\begin{theorem}\label{t61}
Let $X$ be a reflexive Banach space with a symmetric basis. Then
\item $(i)$\,\,\, $JF_{X}$ is isometric to $V_{X}^{0}$.
\item $(ii)$\,\,\, $(V_X^{0})^{**}=V_{X}$.
\item $(iii)$\,\,\, On the bounded subsets of $V_{X}^{0}$
the weak topology coincides with the topology of pointwise
convergence.
\end{theorem}
\begin{remark} It is clear that $V_{X}^{0}$ is a subspace of
$C[0,1]$ and further the identity $I:V_{X}^{0}\mapsto C[0,1]$ is a
bounded operator. Next we shall see that any function $f\in
C[0,1]\cap(V_{X}\setminus V_{X}^{0})$ has a remarkable property.
\end{remark}
\begin{definition}\label{d62}
Let $K$ be a compact metric space. Following
\cite{HOR},\cite{Ro} we denote by $D(K)$ the set of all bounded
functions on $K$ which are differences of bounded semicontinuous
functions.
\end{definition}
If $X$ is a separable Banach space and $K=(B_{X^{*}},w^{*})$ then
as it is shown in \cite{HOR} the classical Bessaga- Pelczynski
theorem,\cite{BP}, yields that there exists $x^{**}\in
D(K)\cap(X^{**}\setminus X)$ iff $c_{0}$ embeds into $X$.

Therefore $D(K)$ provides a characterization of the embedding of
$c_{0}$ into Banach spaces. Our intention is to prove the
following.
\begin{theorem}\label{t63}
If $K=(B_{(V_{X}^{0})^{*}},w^{*})$ then
$$
D(K)\cap V_{X}=C[0,1]\cap V_{X}\,\,.
$$
\end{theorem}
As a corollary we obtain the following characterization.
\begin{corollary}\label{c64}
Let $X$ be a reflexive space with $1-$symmetric basis. Then the
following are equivalent
\begin{enumerate}
\item $c_{0}$ is isomorphic to a subspace of $V_{X}^{0}\equiv
JF_{X}$.
\item There exists a function $f\in C[0,1]$ such that $f\in
V_{X}\setminus V_{X}^{0}$.
\end{enumerate}
\end{corollary}
As we have mentioned in the introduction for the proof of the
above stated theorem we shall make use of methods from descriptive
set theory. We start with the following notation and definition.
\begin{notation} (a) Let $K$ be a metric space,
$ f:\; K\to\mathbb{R}$ and $s\in K$. We set
$$
\overline{\lim_{s^{\prime}\to s}}f(s^{\prime})=\inf\{\sup f(V):
V\,\,\,\text{is a neighborhood of}\,s\}\,\,.
$$
(b) For $f$ as above we denote by $Uf$ the upper semicontinuous
envelope of $f$, which alternatively is defined as follows:
$$
Uf(s)=\overline{\lim_{s^{\prime}\to s}}f(s^{\prime})\,\,.
$$
\end{notation}
\begin{definition} Let $K$ be a compact metric space and
$f:K\mapsto\mathbb{R}$ be a bounded function. For each countable
ordinal $\xi$ the function
$osc_{\xi}(f):K\mapsto\mathbb{R}\cup\{\infty\}$ is defined
inductively as follows.
\newline
For $\xi=0$ we set $osc_{\xi}(f)(s)=0\,\,\forall s\in K $.
\newline
If $0<\xi<\omega_{1}$ and $osc_{\xi}f$ has been defined, we first
set
$$
\widetilde{osc}_{\xi+1}f(s)=\overline{\lim_{s^{\prime}\to s}}\{
\vert f(s)-f(s^{\prime})\vert + osc_{\xi}f(s^{\prime})\}\,\,,
$$
and then we set
$$
osc_{\xi+1}f=U\,\widetilde{osc}_{\xi+1}f\,\,.
$$
If $\xi$ is a limit ordinal and for $\zeta<\xi$, $osc_{\zeta}f$
has been defined then we set
$$
\widetilde{osc}_{\xi}f(s)=\sup_{\zeta<\xi} osc_{\zeta}f(s)\,\,,
$$
and finally
$$
osc_{\xi}f=U\,\widetilde{osc}_{\xi}f\,\,.
$$
This completes the inductive definition.
\end{definition}
The family $\{osc_{\xi}f\}_{\xi<\omega_{1}}$ was introduced by
A.Kechris and A.Louveau, \cite{KL}. H. Rosenthal, \cite{Ro},
recognized the key role of this family in the study of non trivial
$w-Cauchy$ sequences in Banach spaces. The definition presented
here is due to H.Rosenthal and is a modification of the original
one. Some recent results related to this family are obtained in
\cite{AK}. The basic property of the family
$\{osc_{\xi}f\}_{\xi<\omega_{1}}$ is described by the next
proposition.
\begin{proposition}\label{p66} \cite{KL}, \cite{Ro}.
Let $K$ be a compact metric space and $f:K\mapsto\mathbb{R}$ be a
bounded function. The following are equivalent.
\newline (a) The function $f$ is a difference of bounded
semicontinuous functions.
\newline (b) For each $\xi<\omega_{1}$, the function $osc_{\xi}f$
is a bounded function.
\end{proposition}
\begin{lemma}\label{l67} Let $f\in C[0,1]\cap V_{X}$ and
$\mathcal{K}$ be the $w^{*}-$compact subset of
$B_{(V_{X}^{0})^{*}}$ defined in (\ref{e61}). Then for all
$\xi<\omega_{1}$ we have that
\begin{equation}\label{e62}
\Vert osc_{\xi}f\Vert_{\infty}\leq\Vert f_{|\mathcal{K}}\Vert=
\Vert f\Vert_{V_{X}}\,\,.
\end{equation}
\end{lemma}
Clearly (\ref{e62}) and Proposition~\ref{p66} yields that $f\in
D(\mathcal{K})$. Before passing to the proof, we state some
abbreviations and notations.
\begin{notation} In the sequel we consider the set
$\mathcal{K}$ endowed with the $weak^{*}$ topology. Let
$s\in\mathcal{K}$. Then
$s=\sum_{i=1}^{\infty}\alpha_{i}(\delta_{t_{i}^{2}}-
\delta_{t_{i}^{1}})$ where $\{(t_{i}^{1},t_{i}^{2})\}_{i=1
}^{\infty}$ is a family of disjoint intervals and
$\Vert\sum_{i=1}^{\infty}\alpha_{i}e_{i}^{*}\Vert_{X^{*}}\leq 1$.
It is obvious that any permutation  of $\N$ yields a new
representation of the vector $s$. For a fixed representation we
denote by
\begin{align*}
s_{|n_{0}}&=\sum_{i=1}^{n_{0}}\alpha_{i}(\delta_{t_{i}^{2}}-
\delta_{t_{i}^{1}})\in\mathcal{K}\,\,,\\
s_{|>n_{0}}&=
\sum_{i=n_{0}+1}^{\infty}\alpha_{i}(\delta_{t_{i}^{2}}-
\delta_{t_{i}^{1}})\in\mathcal{K}\,\,,\\
s_{|(n_{0},n_{1}]}&=\sum_{i=n_{0}+1}^{n_{1}}
\alpha_{i}(\delta_{t_{i}^{2}}-
\delta_{t_{i}^{1}})\in\mathcal{K}\,\,.
\end{align*}
Moreover, observe that if $(s_{k})_{k}$, $s\in\mathcal{K}$ with
$s_{k}\to s$, $s=
\sum_{i=1}^{\infty}\alpha_{i}(\delta_{t_{i}^{2}}-
\delta_{t_{i}^{1}})$ then for every $n_{0}\in\N$ there exists a
representation of
$s_{k}=\sum_{i=1}^{\infty}\alpha_{i,k}(\delta_{t_{i,k}^{2}}-
\delta_{t_{i,k}^{1}})$ such that for every $i\leq n_{0}$,
$\lim_{k\to\infty}\alpha_{i,k}=\alpha_{i}$,
$\lim_{k\to\infty}t_{i,k}^{1}=t_{i}^{1}$,
$\lim_{k\to\infty}t_{i,k}^{2}=t_{i}^{2}$\,.
\end{notation}
We pass to the following
\begin{lemma}\label{l68} For every $f\in V_{X}$ and every $s\in
\mathcal{K}$, $\lim f(s_{|>n})=0$\,.
\end{lemma}
This is an immediate consequence of the property $\lim_{n}\Vert
s_{|>n}\Vert=0$.
\begin{lemma}\label{l69} Let $f\in C[0,1]\cap V_{X}$ and
$(s_{k})_{k}$, $s$ in $\mathcal{K}$ such that $s_{k}\rightarrow
s$. Then for every $n_{0}\in\N$, $\e>0$ there exists $k\in\N$ such
that for $0\leq n_{1}<n_{2}\leq n_{0}$
$$
\vert f(s_{|(n_{1},n_{2}]})-f(s_{k}{}_{|(n_{1},n_{2}]})\vert
<\e\,\,.
$$
\end{lemma}
This is also easy and follows from the convergence of
$(s_{k})_{k}$ to $s$ and the continuity of the function $f$. It is
worth noticing that this is the only point where the continuity of
the function $f$ is used.
\begin{proof}[Proof of Lemma~\ref{l67}] The proof follows from the
next inductive hypothesis.

For every $s\in\mathcal{K}$, $\e>0$, $\delta>0$ and $n_{0}\in\N$,
if $s=\sum_{i=1}^{\infty}\alpha_{i}(\delta_{t_{i}^{1}}-
\delta_{t_{i}^{2}})$ and $osc_{\xi}f(s)>\delta$, there exist
$s^{\prime}\in\mathcal{K}$ and $n_{1}>n_{0}$ such that
\begin{flushleft}
$(i)\quad\quad\,\text{For}\,\,0\leq m_{1}<m_{2}\leq
n_{0},\quad\text{it holds}\,\,\,\,\vert
f(s_{|(m_{1},m_{2}]})-f(s^{\prime}_{|(m_{1},m_{2}]})\vert<\e$\,.
\\
$(ii)\,\quad\quad\vert f(s^{\prime}_{|>n_{1}})\vert>\delta$\,.
\end{flushleft}
A proof of the inductive hypothesis immediately yields a proof of
the lemma. We proceed by induction.
\newline For $\xi=0$ is trivial.
\newline Assume that for some $\xi<\omega_{1}$ the inducive
hypothesis has been established. Let $s\in\mathcal{K}$, $\e>0$,
$\delta>0$ and $n_{0}\in\N$ such that $osc_{\xi+1}f(s)>\delta$.
>From the definition of $osc_{\xi+1}f$ and Lemma~\ref{l69}, we
choose $\tilde{s}$ such that
\begin{align}
&\widetilde{osc}_{\xi+1}f(\tilde{s})>\delta\,\,, \label{e63}\\
\intertext{and for $0\leq m_{1}<m_{2}\leq n_{0}$}
\vert f(s_{|(m_{1},m_{2}]}&)-
f(\tilde{s}_{|(m_{1},m_{2}]})\vert<\frac{\e}{4}\,\,.\label{e64}
\end{align}
Next we choose  a sequence $(s_{k})_{k}$ such that
\\
$(i)\,\,\,\quad\quad s_{k}\to \tilde{s}$\,. \\
$(ii)\,\quad\quad
\lim_{k}\{\vert f(\tilde{s})- f(s_{k})\vert+osc_{\xi}f(s_{k})\}=
\widetilde{osc}_{\xi+1}f(\tilde{s})$\,. \\
$(iii)\,\quad\quad
\lim_{k}\vert f(\tilde{s})- f(s_{k})\vert=\alpha$\,\,. \\
$(iv)\,\quad\quad \lim_{k} osc_{\xi}f(s_{k})=\beta$\,.\\
Assume that
both $\alpha, \beta$ are positive (If at least one of them is
equal to zero the proof is simpler). Set
$\delta_{1}=\widetilde{osc}_{\xi+1}f(\tilde{s})-\delta>0$ and
since $\alpha+\beta=\widetilde{osc}_{\xi+1}f(\tilde{s})$ there
exist $0<c_{1}<\alpha$, $0<c_{2}<\beta$ such that
$c_{1}+c_{2}>\delta+\frac{3\delta_{1}}{4}$. Next choose
$n_{1}\in\N$ such that
\begin{equation}\label{e65}
\vert
f(\tilde{s}_{|>n})\vert<\frac{\alpha-c_{1}}{16}\,\,\,\text{for
all}\,\,\, n>n_{1}\,\,.
\end{equation}
Further choose $k\in\N$ such that
\begin{align}
\text{For $0\leq m_{1}<m_{2}\leq n_{1}$},\,\,\quad &\vert
f(\tilde{s}_{|(m_{1},m_{2}]})- f(s_{k}{}_{|(m_{1},m_{2}]})\vert
<\min\{\frac{\e}{4},\frac{\alpha-c_{1}}{8}\}\label{e66}\,\,,\\
\vert f(\tilde{s})-& f(s_{k})\vert > c_{1}+\frac{\alpha-c_{1}}{2}
\label{e67}\,\,,\\
&osc_{\xi}f(s_{k})>c_{2}\,\,.\label{e68}
\end{align}
\noindent For this $s_{k}$ we choose $n_{2}>n_{1}$ such that for
all $n\geq n_{2}$,
\begin{equation}\label{e69}
\vert f(s_{k}{}_{|>n})\vert<\frac{\alpha-c_{1}}{8}\,\,.
\end{equation}
Applying the inductive hypothesis for the ordinal $\xi$, the
element $s_{k}$, $\min\{\frac{\e}{4},\frac{\alpha-c_{1}}{8}\}$,
$c_{2}$ and $n_{2}$ we obtain $s^{\prime\prime}\in\mathcal{K}$ and
$n_{3}>n_{2}$ such that
\begin{align}
\text{for $0\leq m_{1}<m_{2}\leq n_{2}$},&\quad \vert
f(s_{k}{}_{|(m_{1},m_{2}]})-f(s^{\prime\prime}_{|(m_{1},m_{2}]})
\vert <\min\{\frac{\e}{4},\frac{\alpha-c_{1}}{8}\}\,\,,\label{e610}\\
&f(s^{\prime\prime}_{|>n_{3}})>c_{2}\,\,.\label{e611}
\end{align}
For the element $s^{\prime\prime}$ we have the following
estimates.
\newline
\indent For $0\leq m_{1}<m_{2}<n_{0}$, (\ref{e64}), (\ref{e66})
and (\ref{e610}) yield that
$$
\vert f(s_{|(m_{1},m_{2}]})-
f(s^{\prime\prime}_{|(m_{1},m_{2}]})\vert<\e\,\,.
$$
Next observe that (\ref{e65}), (\ref{e66}), (\ref{e67}) and
(\ref{e69}) yield
\begin{align*}
\vert f(s_{k}{}_{|(n_{1},n_{2}]})\vert & >\\
&>\vert f(\tilde{s})-f(s_{k})\vert-\vert f(\tilde{s}_{|n_{1}})-
f(s_{k}{}_{|n_{1}})\vert-\vert f(\tilde{s}_{|>n_{1}})\vert-\vert
f(s_{k}{}_{|>n_{2}})\vert
\\
& > c_{1}+\frac{\alpha-c_{1}}{2}-\frac{\alpha-c_{1}}{8}-
\frac{\alpha-c_{1}}{8}-\frac{\alpha-c_{1}}{8}
\\
& =c_{1}+\frac{\alpha-c_{1}}{8}\,\, \intertext{and from
(\ref{e610}) we get} &\quad\quad\quad\quad\vert
f(s^{\prime\prime}_{|(n_{1},n_{2}]})\vert>
c_{1}+\frac{\alpha-c_{1}}{8}-\frac{\alpha-c_{1}}{8}=c_{1}\,\,.
\end{align*}
Observe that
$$
s^{\prime\prime}_{|>n_{1}}=s^{\prime\prime}_{|(n_{1},n_{2}]}+
s^{\prime\prime}_{|(n_{2},n_{3}]}+s^{\prime\prime}_{|>n_{3}}\,\,,
$$
and since $f$ is an affine function on $\mathcal{K}$ we obtain
that there exist $\e_{1}$, $\e_{2}$, $\e_{3}\in\{-1,1\}$ such that
$$
f(\e_{1}s^{\prime\prime}_{|(n_{1},n_{2}]})\geq 0,\quad
f(\e_{2}s^{\prime\prime}_{|(n_{2},n_{3}]})\geq 0,\quad
f(\e_{3}s^{\prime\prime}_{|>n_{3}]})\geq 0\,\,.
$$
It is easy to check that the element
$$
s^{\prime}=s^{\prime\prime}_{|n_{1}}+
\e_{1}s^{\prime\prime}_{|(n_{1},n_{2}]}+
\e_{2}s^{\prime\prime}_{|(n_{2},n_{3}]}+
\e_{3}s^{\prime\prime}_{|>n_{3}}
$$
satisfies the conclusion of the inductive hypothesis.

If $\xi$ is a limit ordinal and for some $s\in\mathcal{K}$, such
that $\widetilde{osc}_{\xi}f(s)>\delta$, then from the definition
of $osc_{\xi}f$ and Lemma~\ref{l69}, there exist $\zeta<\xi$ and
$\tilde{s}$ such that
\begin{align*}
&osc_{\zeta}f(\tilde{s})>\delta\,\,,\\
\text{and for}\,\,\, 0\leq m_{1}<m_{2}\leq n_{0}&,\,\,\quad \vert
f(s_{|(m_{1},m_{2}]})-
f(\tilde{s}_{|(m_{1},m_{2}]})\vert<\frac{\e}{2}\,\,.
\end{align*}
Applying the inductive hypothesis for $\zeta$ the element
$\tilde{s}$, $\frac{\e}{2}$, $\delta$ and $n_{0}$ we get an
element $s^{\prime}$, which it is easily seen that satisfies the
conclusion of the lemma. The proof is complete.
\end{proof}
As a consequence of the previous lemma we obtain the following
\begin{proposition}\label{p610} Let $f\in C[0,1]\cap V_{X}$. Then
$f$ is a difference of bounded semicontinuous functions defined on
$L=(B_{(V_{X}^{0})^{*}},w^{*})$.
\end{proposition}
\begin{proof} If $f\in V_{X}^{0}$ the conclusion trivially holds.
Hence assume that $f\in V_{X}\setminus V_{X}^{0}$. From
Lemma~\ref{l67} we obtain that $f\in D(\mathcal{K})$. Since
$\ell_{1}$ does not embed in $V_{X}^{0}$, Odell- Rosenthal's
theorem \cite{OR}, yields that there exists a bounded sequence
$(f_{n})_{n}$ in $V_{X}^{0}$ converging $weak^{*}$ to $f$. Hence
$(f_{n}{}_{|\mathcal{K}})_{n}\subset C(\mathcal{K})$ and converges
pointwise to $f_{|\mathcal{K}}$. From \cite{HOR} we obtain that
there exists $(g_{n})_{n}$ block of convex combinations of
$(f_{n})_{n}$ such that $(g_{n}{}_{|\mathcal{K}})_{n}$ is
isomorphic to the summing basis of $c_{0}$. Since $\mathcal{K}$
norms $V_{X}^{0}$, we obtain that $(g_{n})_{n}$ remains isomorphic
to the summing basis and converges $weak^{*}$ to $f$. Hence $f\in
D(L).$
\end{proof}
\noindent This proposition and known results yield the following.
\begin{proposition}\label{p611}
Let $Y$ be a non reflexive subspace of $V_{X}^{0}$. Assume that
there exists $f\in\overline{Y}^{w^{*}}$ such that $f\in
C[0,1]\cap(V_{X}\setminus V_{X}^{0})$. Then $c_{0}$ is isomorphic
to a subspace of $Y$.
\end{proposition}
Up to this point we have proved one direction of the
Theorem~\ref{t63}. For the inverse part we need the following
lemma.
\begin{lemma}\label{l612}
Let $(f_{n})_{n}$ be a normalized weakly null sequence in
$V_{X}^{0}$. Let $\varepsilon>0$, $([\alpha_{n},\beta_{n}])_{n}$
be a sequence of intervals such that
$\lim_{n}(\beta_{n}-\alpha_{n})=0$ and $\vert
f_{n}(\alpha_{n})-f_{n}(\beta_{n})\vert >\varepsilon,$ for all
$n\in\mathbb{N}.$ Then there exists a subsequence
$(f_{n_{l}})_{l}$ of $(f_{n})_{n}$ which admits a lower
$X$-estimate, i.e. there exists $c>0$ such that
$$
c\,\Vert\sum_{l}\alpha_{l}e_{l}\Vert_{X}\leq \Vert\sum_{l}
\alpha_{l}f_{n_{l}}\Vert_{V_{X}^{0}}\,\,\text{for all}\,\,
(\alpha_{l})_{1}^{n}\subset\mathbb{R}\,\,.
$$
\end{lemma}
\begin{proof}Passing to a subsequence, we may assume that
$(\alpha_{n})_{n}$, $(\beta_{n})_{n}$ are monotone sequences, such
that $\lim_{n}\alpha_{n}=\alpha=\lim_{n}\beta_{n}$. Using that
$(f_{n})_{n}$ is weakly null and the fact that finite subsets of
$C(0,1)$ are equicontinuous, by a diagonal process we obtain a
subsequence $([c_{n},d_{n}])_{n\in L}$, $L\subset \mathbb{N}$,
such that
\begin{enumerate}
\item $[c_{n},d_{n}]\subset [\alpha_{n},\beta_{n}]$ for all
$n\in L$.

\smallskip
\item For $n_{1},n_{2}\in L$, $n_{1}<n_{2}$, we have that
$[c_{n_{1}},d_{n_{1}}]\cap
[\alpha_{n_{2}},\beta_{n_{2}}]=\emptyset $.

\smallskip
\item $\vert f_{n}(d_{n})-f_{n}(c_{n})\vert>\frac{\varepsilon}{4}$
, for all $n\in L.$

\smallskip
\item For all $n\in L$,
$\sum_{\substack{k<n \\ k\in L}}\vert
f_{k}(d_{n})-f_{k}(c_{n})\vert <\frac{\varepsilon}{2\cdot
8^{n+1}}$, from the equicontinuity.

\smallskip
\item For all $n\in L$,
$\sum_{L\ni k>n}\vert f_{k}(d_{n})-f_{k}(c_{n})\vert <
\frac{\varepsilon}{2\cdot 8^{n+1}}$, since $(f_{n})_{n}$ is weakly
null.
\end{enumerate}
Set $(f_{n})_{n\in L}=(f_{n_{l}})_{l\in \mathbb{N}}.$ We prove
that the subsequence $(f_{n_{l}})_{l\in \mathbb{N}}$ admits a
lower $\frac{\varepsilon}{8}-X$ estimate. Indeed,
\begin{align*}
\Vert\sum_{l=1}^{k}\alpha_{l}f_{n_{l}}&\Vert_{V_{X}^{0}}\geq \Vert
\sum_{j=1}^{k}\left( \sum_{l=1}^{k}\alpha_{l}f_{n_{l}}(d_{j})-
\sum_{l=1}^{k}\alpha_{l}f_{n_{l}}(c_{j})\right)e_{j}\Vert_{X}\\
\geq & \Vert\sum_{l=1}^{k} \alpha_{l}
\left(f_{n_{l}}(d_{n_{l}})-f_{n_{l}}(c_{n_{l}})\right)e_{l}\Vert_{X}-
\max_{1\leq l\leq k}\vert\alpha_{l}\vert\sum_{j=1}^{k}
\sum_{l\not=j}\vert f_{n_{l}}(d_{n_{j}})-f_{n_{l}}(c_{n_{j}})\vert\\
\geq &\frac{\varepsilon}{4}
\Vert\sum_{l=1}^{k}\alpha_{l}e_{l}\Vert_{X}- \max_{1\leq l\leq
k}\vert\alpha_{l}\vert\sum_{j=1}^{k}\left((\sum_{l<j}+\sum_{l>j})
\vert f_{n_{l}}(d_{n_{j}})-f_{n_{l}}(c_{n_{j}})\vert\right)
\\
\geq & \frac{\varepsilon}{4}
\Vert\sum_{l=1}^{k}\alpha_{l}e_{l}\Vert_{X}-(\max_{1\leq l\leq
k}\vert\alpha_{l}\vert)\frac{\varepsilon}{8}\geq
\frac{\varepsilon}{8}\Vert\sum_{l=1}^{k}\alpha_{l}e_{l}\Vert_{X}\,\,.
\end{align*}
This completes the proof of the lemma.
\end{proof}
\begin{proposition}\label{p613}
Let  $f\in V_{X}\setminus C[0,1]$. Then for every bounded sequence
$(f_{n})_{n}$ converging $weak^{*}$ to $f$ there exists a
subsequence $(f_{n_{j}})_{j}$ such that
$(f_{n_{2j+1}}-f_{n_{2j}})_{j}$ has a lower $X$ estimate.
\end{proposition}
\begin{proof}
Since $f$ is not continuous there exists $t\in [0,1]$,
$(t_{k})_{k}$ converging to $t$ and $\e>0$ such that $\vert
f(t)-f(t_{k})\vert>\e$. Therefore there exist subsequences
$(f_{n_{j}})_{j}$ and $(t_{k_{j}})_{j}$ such that $\vert
f_{n_{j}}(t)-f_{n_{j}}(t_{k_{j}})\vert>\e$. Further since each
$f_{n_{j}}$ is continuous we may assume that $\vert
f_{n_{j}}(t_{k_{j+1}})-f_{n_{j}}(t)\vert <\frac{\e}{4}$. Clearly
$(f_{n_{2j+1}}-f_{n_{2j}})_{j}$ satisfies the assumption of
Lemma~\ref{l612}, and hence has a further subsequence with lower
$X-$estimate.
\end{proof}
As a corollary of Proposition~\ref{p613} we get the following.
\begin{corollary}\label{c614}
If $f\in V_{X}\setminus C[0,1]$, then $f\not\in D(L)$, where
$L=(B_{(V_{X}^{0})^{*}},w^{*})$.
\end{corollary}
\begin{proof}[Proof of Theorem~\ref{t63}]
Follows from Proposition~\ref{p610} and Corollary~\ref{c614}.
\end{proof}
We pass now to give a criterion for upper $X-$estimate.
\begin{lemma}\label{l615}
Let $(f_{n})_{n}$ be a normalized weakly null sequence in
$V_{X}^{0}$ such that $\Vert f_{n}\Vert_{\infty}\geq C$ for every
$n\in\N$. Then for every $\e>0$ there exist $M\in [\N]$ and
$t_{1},\ldots,t_{k(\e)}$ points such that, for every $\delta>0$
there exists $n_{0}$, such that
$$
\Vert f_{n}{}_{|[0,1]\setminus \cup_{i=1}^{k(\e)}B(t_{i},\delta)}
\Vert_{\infty}<\e\,\,\text{for every}\,\,\,\, n>n_{0}, n\in M\,\,.
$$
\end{lemma}
\begin{proof}
Suppose that the conclusion does not hold. Inductively we choose
$M_{1}\supset\ldots\supset M_{k}$, $(t_{n}^{j})_{n\in M_{j}}$,
$s_{j}$, $\delta_{j}>0$, $j\leq k$ such that
\begin{enumerate}
\item $\vert f_{n}(t_{n}^{j})\vert >\e$ for every $n\in M_{j}$.
\item $\lim_{n\in M_{j}}t_{n}^{j}=s_{j}$ for every $j\leq k$
and $s_{i}\not=s_{j}$ for $i\not= j$.
\item $\Vert f_{n}{}_{|[0,1]\setminus \cup_{i=1}^{j}
B(s_{i},\delta_{j})}\Vert_{\infty}>\e$ for every $n\in M_{j+1}$.
\end{enumerate}
Passing to a further subset of $M_{k}$ we may assume  that for
every $n\in M_{k}$ it holds, $\vert
f_{n}(s_{j})\vert<\frac{\e}{2}$ for every $j\leq k$. Then for
$n\in M_{k}$ sufficiently large, we have that
$$
\Vert f_{n}\Vert_{V_{X}^{0}}\geq
\frac{\e}{2}\Vert\sum_{i=1}^{k}e_{i}\Vert_{X}\,\,,
$$
a contradiction for large $k$, since
$\Vert\sum_{i=1}^{k}e_{i}\Vert_{X}\to\infty$.
\end{proof}
\begin{lemma}\label{l616}
Let $(f_{n})_{n}$ be a normalized weakly null sequence in
$V_{X}^{0}$ such that $\Vert f_{n}\Vert_{\infty}\geq C$ for every
$n\in\N$. Then for every $\e>0$ there exists $M\in [\N]$ such the
sets $U_{n}=\{t :\vert f_{n}(t)\vert\geq\e\}$ are pairwise
disjoint.
\end{lemma}
\begin{proof} Let $\e>0$. Passing to a subsequence,
by Lemma~\ref{l615}, we may assume that there exists $k(\e)$
points $t_{1},\ldots,t_{k(\e)}$  such that
$$
\Vert f_{n}{}_{|[0,1]\setminus \cup_{i=1}^{k(\e)}B(t_{i},\delta)}
\Vert_{\infty}<\e \quad\quad \text{for every}\,\,n\in\N\,\,,
$$
and also $\vert f_{n}(t_{i})\vert<\frac{\e}{4}$, for every
$n\in\N$ and every $i=1,\ldots,k(\e)$. Set $f_{n_{1}}=f_{1}$.
Since $f_{n_{1}}$ is continuous, we find for every $i\leq k(\e)$ a
neighborhood $B_{1}(t_{i},\delta_{1})\subset B(t_{i},\delta)$ of
$t_{i}$ such that $\vert
f_{n_{1}}(B_{1}(t_{i},\delta_{1})\vert<\frac{\e}{4}$. Set
$U_{1}=\{t: \vert f_{n_{1}}(t)\vert>\e\}$. Then $U_{1}\cap
(\cup_{i=1}^{k(\e)} B_{1}(t_{i},\delta_{1}))=\emptyset$. Using
Lemma~\ref{l615}, we pass to a subsequence $(f_{n})_{n\in M_{2}}$
such that
$$
\Vert f_{n}{}_{|[0,1]\setminus
\cup_{i=1}^{k(\e)}B_{1}(t_{i},\delta_{1})}\Vert_{\infty}
<\e\,\quad\quad\text{for every}\,\,\,n\in M_{2}.
$$
Set $f_{n_{2}}=f_{\min M_{2}}$. Since $f_{n_{2}}$ is continuous,
we find for every $i\leq k(\e)$ a neighborhood
$B_{2}(t_{i},\delta_{2})\subset B_{1}(t_{i},\delta_{1})$ of
$t_{i}$ such that $\vert
f_{n_{2}}(B_{2}(t_{i},\delta_{2}))\vert<\frac{\e}{4}$. Set
$U_{2}=\{t: \vert f_{n_{2}}(t)\vert>\e\}$. Then $U_{2}\subset
\cup_{i=1}^{k(\e)}( B_{1}(t_{i},\delta_{1})\setminus
B_{2}(t_{i},\delta_{2}))$ and $U_{2}\cap U_{1}=\emptyset$.
Continuing in the same manner we get the desired subsequence.
\end{proof}
The following corollary of the above lemmas seems to be of
independent interest.
\begin{corollary}[Splitting lemma]\label{l617}
Let $(f_{n})$ be a normalized weakly null sequence in $V_{X}^{0}.$
Then there exists a subsequence $(f_{n})_{n\in M}$ of
$(f_{n})_{n}$ such that $f_{n}=g_{n}+h_{n}$ for every $n\in M$,
$\lim_{n}\Vert g_{n}\Vert_{\infty}=0$ and $\lim_{n}\mu({\rm
supp}h_{n})=0$\,.
\end{corollary}
\begin{proof}
For the sequence $(f_{n})_{n}$, we may assume
that there exists a constant $C>0$ such that $\Vert
f_{n}\Vert_{\infty}\geq C$ for all $n\in\N$. Otherwise there
exists a subsequence $(f_{n})_{n\in M}$ of $(f_{n})_{n}$ such that
$\lim_{n\in M}\Vert f_{n}\Vert_{\infty}=0$, and the result follows
immediately, setting $g_{n}=f_{n}$ and $h_{n}=0$.

Applying Lemma~\ref{l615} inductively for $\e=\frac{1}{2^{j}}$ we
get a decreasing sequence $(M_{j})_{j}$ of infinite subsets of
$\N$ and subsequences $(f_{n})_{n\in M_{j}}$ and $t^{j}_{1},\ldots
,t^{j}_{k(j)}$ points such that
$$
\Vert f_{n}{}_{|[0,1]\setminus\cup_{i=1}^{k(j)}
B(t^{j}_{i},\delta_{j})}\Vert_{\infty} <
\frac{1}{2^{j}}\,\,\text{for every}\,\,\,n\in M_{j}\,\,,
$$
where $\delta_{j}=\frac{k(j)}{2^{j}}.$ Let $n_{j}=\min M_{j}.$ For
every $i\leq k(j)$, let $g^{j}_{i}$ be a linear function defined
in $[t^{j}_{i}-\delta_{j}, t^{j}_{i}+\delta_{j}]$ with endpoints
$f(t^{j}_{i}-\delta_{j})$, $f(t^{j}_{i}+\delta_{j})$. Define
$g_{j}(t)=f_{n_j}(t)$  for every
$t\not\in\cup_{i=1}^{k(j)}B(t^{j}_{i},\delta_{j})$, while
$g_{j}(t)=g^{j}_{i}(t)$ if $t\in B(t^{j}_{i},\delta_{j})$, $i\leq
k(j)$. We also set $h_{j}=f_{n_{j}}-g_{j}$. It is easy to see that
$g_{j}$, $h_{j}$ have the properties we claim.
\end{proof}
\begin{remark} The isomorphic structure of the subspaces of
$JF_{X}(\Omega)$ remains unclear, even in the case of James
function space. We have been informed by E.Odell that in the Ph.D.
Thesis of his student S.Buechler \cite{Bu}, is included the
following result. Every normalized weakly null sequence
$(x_{n})_{n}$ and $\varepsilon>0$ there exists a subsequence
$(x_{n})_{n\in M}$ admitting an $\sqrt{2+\varepsilon}$ upper
$\ell_{2}$-estimate. We present a slightly more general result.
\end{remark}
\begin{definition}Let $X$ be a Banach space with 1-symmetric basis
$(e_{i})_{i}$. The space $X$ has the block dominated property, if
there exists $C>0$ such that for every normalized block sequence
$(u_{i})_{i}$ we have that
$$
\Vert\sum_{i}\alpha_{i}u_{i}\Vert \leq
C\Vert\sum_{i}\alpha_{i}e_{i}\Vert\,\,.
$$
\end{definition}
\noindent As a consequence of Lemmas~\ref{l615}, \ref{l616} we get
the following theorem.
\begin{theorem}\label{t619}
Let $X$ have the block dominated property. Then every normalized
weakly null sequence $(f_{n})$ in $V_{X}^{0}$, has a further
subsequence $(f_{n})_{n\in M}$ which admits un upper $X-$estimate.
\end{theorem}
\begin{proof}
We distinguish two cases for the sequence $(f_{n})_{n}$. The proof
is almost identical in the two cases. We present the proof of the
first case, and we shall indicate at the end the modification for
the second case.
\begin{case}1. There exists $C>0$ such that
$\Vert f_{n}\Vert_{\infty}\geq C$ for all $n\in N$.
\end{case}
Applying inductively Lemma~\ref{l616} we choose a subsequence
$(f_{n_{j}})_{j}$ of $(f_{n})_{n}$, $(\delta_{j})_{j}$,
$(\e_{j})_{j}$ sequences of real numbers such that
\begin{enumerate}
\item $\tau(f_{n_{j}},\mathcal{P})<\frac{\e}{2^{j}}$ for every
$\mathcal{P}$ with $\delta(\mathcal{P})\leq\delta_{j}$.
\item $\e_{1}>\ldots>\e_{k}>\ldots$ and
$\e_{j+1}<\e\frac{\delta_{j}}{2^{j}}$ for every $j>1$.
\item The sets $U_{i}=\{t :\vert f_{n_{i}}(t)\vert
\geq\frac{\e_{j}}{2}\}$ are pairwise disjoint for every $i\geq j$.
\end{enumerate}
The subsequence $(f_{n_{j}})_{j}$ admits an $(3C+2\e)$-upper
$X-$estimate i.e.
$$
\Vert\sum_{i=1}^{m}\alpha_{i}f_{n_{i}}\Vert_{V_{X}^{0}}\leq
(3C+2\e)\Vert\sum_{i=1}^{m}\alpha_{i}e_{i}\Vert_{X}
\quad\quad\text{for
every}\,\,\,\,\{\alpha_{i}\}_{i=1}^{m}\subset\mathbb{R}\,.
$$
Indeed, let $\mathcal{Q}=\{t_{i}\}_{i=1}^{r}$ be a partition of
$[0,1]$. Consider the partition of $\mathcal{Q}$ into
$Q_{j}=\{t_{i}: \delta_{j}<t_{i+1}-t_{i}\leq\delta_{j-1}\}$, where
$\delta_{0}=1$, $j\leq m$. Property (1) implies that
$$
\Vert\sum_{j=1}^{m}\sum_{t_{i}\in Q_{j}}\left(\sum_{k<j}
\alpha_{k}(f_{n_{k}}(t_{i+1})-f_{n_{i}}(t_{i}))
\right)e_{i}\Vert_{X}\leq
\sum_{k=1}^{m}\tau(f_{n_{k}},\cup_{j>k}Q_{j})\leq\e\,\,.
$$
Also for fixed $j$ and $t_{i}\in Q_{j}$, property (3) implies that
for at most three  $k\geq j$, $k_{1}^{i},k_{2}^{i},k_{3}^{i}$,
$\vert f_{n_{k}}(t_{i+1})-f_{n_{k}}(t_{i})\vert\geq \e_{j+1}$. For
every $t_{i}$, let $k_{t_{i}}$ be the $k_{j}^{i}$, $j\leq 3$,
which realize $\max\{\vert\alpha_{k_{j}^{i}}(
f_{n_{k_{j}^{i}}}(t_{i+1})-f_{n_{k_{j}^{i}}}(t_{i}))\vert: j\leq
3\}$. Therefore
\begin{align*}
\Vert\sum_{j=1}^{m}\sum_{t_{i}\in Q_{j}}&\left(\sum_{k\geq
j}\alpha_{k}(f_{n_{k}}(t_{i+1})-f_{n_{k}}(t_{i}))\right)
e_{i}\Vert\\
&\leq 3\Vert\sum_{k=1}^{m}\alpha_{k}
\sum_{j=1}^{m}\sum_{\substack{t_{i}\in Q_{j}:\\k_{t_{i}}=k }}
\left(f_{n_{k}}(t_{i+1})-f_{n_{k}}(t_{i})\right)e_{i}\Vert
+\sum_{j=1}^{m}\# Q_{j}\cdot\e_{j+1}\\
&\leq 3C\Vert\sum_{k=1}^{m}\alpha_{k}e_{k}\Vert +\e\,\,,
\end{align*}
since $X$ has the block dominated property.
\begin{case}2 There exists a subsequence $(f_{n})_{n\in M}$ of
$(f_{n})_{n}$ such that $\lim_{n\in M}\Vert
f_{n}\Vert_{\infty}=0$.
\end{case}
In this case we proceed as in Case 1, choosing a further
subsequence $(f_{n_{j}})_{j}$ of $(f_{n})_{n\in M}$,
$(\delta_{j})_{j}$, $(\e_{j})_{j}$ sequences of real numbers such
satisfying $(1)$, $(2)$ as above, and $(3)$ is replaced by

$(3^{\prime})$ For every $j\in\N$, it holds that $\Vert
f_{n_{i}}\Vert_{\infty}<\frac{\e_{j+1}}{2}$ for every $i>j$\,.

Then, following the arguments of Case 1, we easily seen that
$(f_{n_{j}})_{j}$ admits an $(C+2\e)$ upper estimate.
\end{proof}
The following two results follows from our previous work.
\begin{proposition}\label{p620}
Let $X$ have the block dominated property and $(f_{n})_{n}$ be a
normalized $w-Cauchy$ sequence in $V_{X}^{0}$ which converges
$weak^{*}$ to $f\in V_{X}\setminus C[0,1]$. Then there exists a
subsequence $(f_{n_{k}})_{k}$ of $(f_{n})_{n}$ such that
$(f_{n_{2k+1}}-f_{n_{2k}})_{k}$ is equivalent to the basis of $X$.
\end{proposition}
\begin{proof}
By Proposition~\ref{p613} and the discontinuity of $f$ there exist
$\e>0$, and a subsequence $(f_{n})_{n\in M}$ of $(f_{n})_{n}$ such
that $(f_{2n+1}-f_{2n})_{n\in M}$ has a lower estimate, and $\Vert
f_{2n+1}-f_{2n}\Vert_{\infty}\geq\e$ for all $n\in M$. The
sequence $(f_{2n+1}-f_{2n})_{n\in M}$ satisfies the assumptions of
Theorem~\ref{t619}, and therefore there exists a further
subsequence $(f_{n_{2k+1}}-f_{n_{2k}})_{k}$ which admits an upper
$X-$estimate.
\end{proof}
Proposition~\ref{p611} and Proposition~\ref{p620} immediately
yield the following.
\begin{theorem}\label{t621}
Let $X$ have the block dominated property and $Y$ be a non
reflexive subspace of $V_{X}^{0}$. Then $Y$ contains
isomorphically $c_{0}$ or $X$.
\end{theorem}
The following result answer partially the question on the
structure of the subspaces of $JF_{X}$, $X$ having the block
dominated property.
\begin{proposition}\label{p622}
Let $X$ have the block dominated property, and $(x_{n})_{n}$ be a
normalized $\delta-$separated (i.e $\Vert
x_{n}-x_{m}\Vert>\delta>0$) sequence in $JF_{X}\cap L^{1}$, such
that $(\Vert x_{n}\Vert_{L^{1}})_{n}$ is bounded. Then the
subspace $Y$ generated by $(x_{n})_{n}$ contains isomorphically
$X$.
\end{proposition}
In the proof we shall use the following lemma, which holds for any
reflexive Banach space with 1-symmetric basis.
\begin{lemma}\label{l623}
Let $(x_{n})_{n}$ be a normalized weakly null sequence in
$JF_{X}\cap L^{1}(\mu)$. Let $\varepsilon>0$,
$([\alpha_{n},\beta_{n}])_{n}$ be a sequence of intervals such
that $\lim_{n}(\beta_{n}-\alpha_{n})=0$ and
$\vert\int_{\alpha_{n}}^{\beta_{n}}x_{n}\vert >\varepsilon$ for
all $n\in\mathbb{N}.$ Then there exists a subsequence
$(x_{n_{l}})_{l}$ of $(x_{n})_{n}$ which admits a lower
$X$-estimate, i.e there exists $c>0$ such that
$$
c\,\Vert\sum_{l}\alpha_{l}e_{l}\Vert_{X}\leq \Vert\sum_{l}
\alpha_{l}x_{n_{l}}\Vert_{JF_{X}}\,\,\text{for all}\,\,
(\alpha_{l})_{1}^{n}\subset\mathbb{R}\,\,.
$$
\end{lemma}
The proof of this lemma follows immediately from Lemma~\ref{l612},
considering $JF_{X}$ as the space $V_{X}^{0}$.
\begin{proof}[Proof of Proposition~\ref{p622}]
Since $\ell_{1}$ does not embed into $JF$, passing to a
subsequence and taking differences of successive terms, we assume
that $(x_{n})_{n}$ is normalized weakly null sequence.
\begin{claim}
There exist $\varepsilon>0$, $M$ an infinite subset of $\mathbb{N}$
and a sequence $([\alpha_{n},\beta_{n}])_{n\in M}$ of intervals
such that
$$
\vert\int_{\alpha_{n}}^{\beta_{n}}x_{n}d\mu\vert>\varepsilon\,\,
\text{for all}\,\,n\in M\,\,\,\,\text{and also}\,\,\,
\lim_{n}(\beta_{n}-\alpha_{n})=0\,.
$$
\end{claim}
\begin{proof}[Proof of the Claim.] Since $(x_{n})_{n}$ is bounded in
$L^{1}$-norm from Lemma~\ref{pe1} (see also proof of Proposition
\ref{l114}), we conclude that there exist $\varepsilon>0$ and a
sequence $([\alpha_{n},\beta_{n}])_{n}$ of intervals, such that
\begin{equation}\label{e612}
\vert\int_{\alpha_{n}}^{\beta_{n}}x_{n}d\mu\vert>\varepsilon\,\,
\text{for all}\,\,n\in\mathbb{N}\,\,.
\end{equation}
If \,\,$\lim_{n}(\beta_{n}-\alpha_{n})=0$, we have finished,
otherwise let's observe the following.

There exist $\varepsilon^{\prime}>0$, $M$ an infinite subset of
$\mathbb{N}$, such that for all $n\in M$, there exists an interval
$[c_{n},d_{n}]$ such that
\begin{enumerate}
\item
$\vert\int_{c_{n}}^{d_{n}}x_{n}d\mu\vert>\varepsilon^{\prime}$ for
all $n\in M$,
\item \,$\lim_{n}(d_{n}-c_{n})=0.$
\end{enumerate}
Indeed, choose monotone subsequences
$(\alpha_{n}$,$\beta_{n})_{n\in M}$ such that
$\alpha_{n}\to\alpha$, $\beta_{n}\to\beta.$ For simplicity assume
that $[\alpha,\beta]\subset [\alpha_{n},\beta_{n}].$ Since
$(x_{n})_{n}$ is weakly null, we choose $ n_{0}\in M $ such that
$$
\vert\int_{\alpha}^{\beta}x_{n}d\mu\vert
<\frac{\varepsilon}{4}\,\,\,\,\,\text{for all} \,\,n\in M,\,
n>n_{0}\,\,.
$$
Clearly (\ref{e612}) implies that for $n\in M $, $n>n_{0}$,
$$
\text{either}\,\,\,\,\vert
\int_{\alpha_{n}}^{\alpha}x_{n}d\mu\vert
>\frac{\varepsilon}{4}\,\,\,\,\, \text{or}\,\,\,\,\,
\vert\int_{\beta}^{\beta_{n}}x_{n}d\mu\vert>\frac{\varepsilon}{4}\,\,.
$$
If the former holds, set $[c_{n},d_{n}]=[\alpha_{n},\alpha]$
otherwise set $[c_{n},d_{n}]=[\beta,\beta_{n}]$ and let
$\varepsilon^{\prime}=\frac{\varepsilon}{4}$. We easily conclude
that $(x_{n})_{n\in M}$, $([c_{n},d_{n}])_{n\in M}$,
$\varepsilon^{\prime}$ satisfy the desired properties. The proof
of the Claim is complete.
\end{proof}
To complete the proof of Proposition~\ref{p622}, Lemma~\ref{l623}
yields a further subsequence $(x_{n_{k}})_{k\in\mathbb{N}}$ which
admits a lower $X$-estimate. From Theorem~\ref{t619},
$(x_{n_{k}})_{k\in\mathbb{N}}$  has a subsequence
$(x_{n_{k}})_{k\in M}$ with an upper $X$-estimate. Clearly
$(x_{n_{k}})_{k\in M}$ is equivalent to the basis of $X.$
\end{proof}
A direct consequence of Proposition~\ref{p622} is the following.
\begin{corollary}\label{c624}
Let $X$ have the block dominated property. Then every subsequence
$(h_{n_{k}})_{k}$ of the Haar system in $JF_{X}$, generates a
subspace containing $X.$
\end{corollary}
In the last part of this section we prove the equivalence between
the Point Continuity Property $(PCP)$, and the non embedding of
$c_{0}$ in $V_{X}^{0}$. We recall that a bounded subset $W$ of a
Banach space $X$ has $PCP$ if for every  $w-$closed subset $A$ of
$W$ and $\e>0$ there exists a relatively weakly open subset $U$ of
$A$ with $diam(U)<\e$. For a comprehensive study of $PCP$ and its
relation with other properties we refer to \cite{b}. We start with
the following.
\begin{lemma}\label{l625} Let $W$ be a bounded subset of
$V_{X}^{0}$. Then for every relatively weakly open $U\subset W$
and every $\e>0$ there exists $g\in U$ such that for every weakly
open neighborhood $V$ of $g$ and every $g^{\prime}\in U\cap V$ we
have that $\Vert g-g^{\prime}\Vert_{\infty}\leq\e$.
\end{lemma}
\begin{proof} Assume on the contrary. There exists a bounded set
$W$ and a relatively weakly open subset $U$ of $W$ such that for
every $g\in U$ there exists a net $(g_{i})_{i\in I}\subset U$ with
$g_{i}\stackrel{w}{\longrightarrow}g$ and $\Vert
g-g_{i}\Vert_{\infty}\geq\e.$ Since this property remains
invariant under translations of $W$ we assume that $0\in U$ and
that for every $g\in W$, $\Vert g\Vert_{V_{X}}\leq C$.
\newline\noindent Choose $n_{0}\in\N$ such that
$$
\Vert\sum_{i=1}^{n_{0}}e_{i}\Vert_{X}\geq\frac{4C}{\e}\,\,.
$$
\begin{claim} There exists $g\in U$ with
$\Vert g\Vert_{V_{X}}>2C$.
\end{claim}
Clearly a proof of this claim derives a contradiction and
completes the proof of the lemma.

\noindent {\it Proof of the Claim.} To prove the claim, we proceed
by induction on $k=0,\ldots, 2n_{0}^{2}$ choosing
\begin{enumerate}
\item A partition $\mathcal{P}_{k}$ of $[0,1]$,
\item $g_{k}\in U$,
\item $[\alpha_{k},\beta_{k}]$, $0\leq\alpha_{k}<\beta_{k}\leq 1$,
\,\,\,\,\,such that the following are fulfilled.
\begin{enumerate}
\item $g_{0}=0$, $\mathcal{P}_{0}=\{0,1\}$, $\alpha_{0}=0$,
$\beta_{0}=1$.
\item
$\mathcal{P}_{0}\subset\mathcal{P}_{1}\subset\ldots
\subset\mathcal{P}_{2n_{0}^{2}}$.
\item $\{\alpha_{k},\beta_{k}\}\subset\mathcal{P}_{k+1}$.
\item For $k=1,\ldots,2n_{0}^{2}$,\,\,$\vert (g_{k}-g_{k-1})(t)\vert
>\e$
\,\,\,\,for all $t\in [\alpha_{k+1},\beta_{k+1}]$.
\item $\vert (g_{k}-g_{k+1})(s)\vert<\frac{\e}{8^{k+1}}$\,\,\,\,
for $s\in\mathcal{P}_{k+1}$.
\item $diam(\sum_{j=0}^{k-1}\vert g_{j}-g_{j+1}\vert([s_{i}^{k+1},
s_{i+1}^{k+1}])<\frac{\e}{8}$, where
$\mathcal{P}_{k+1}=\{0=s_{0}^{k+1}<\ldots<s_{d_{k+1}}^{k+1}=1\}$\,.
\end{enumerate}
\end{enumerate}
Assume that for some $0\leq k<2n_{0}^{2}$, $\mathcal{P}_{j}$,
$g_{j}$, $[\alpha_{j},\beta_{j}]$ have been defined for $0\leq
j\leq k$. Then we first choose $\mathcal{P}_{k+1}$ such that
$\mathcal{P}_{k}\cup\{\alpha_{k},\beta_{k}\}\subset
\mathcal{P}_{k+1}$ and condition $(f)$ is satisfied. The latter is
possible because of the continuity of the function
$\sum_{j=1}^{k-1}\vert g_{j}-g_{j+1}\vert$. Next we choose a net
$(g_{i})_{i\in I}\subset U$ such that
$g_{i}-g_{k}\stackrel{w}{\longrightarrow}0$, $\Vert
g_{i}-g_{k}\Vert_{\infty}>\e$, and clearly there exists
$g_{i_{0}}\equiv g_{k+1}$ satisfying condition $(e)$. Finally
choose $\alpha_{k+1}<\beta_{k+1}$ such that $\vert
g_{k}-g_{k+1}(t)\vert>\e$ for all $t\in
[\alpha_{k+1},\beta_{k+1}]$. Let's observe the following.

\smallskip
$(i)$ For all $k=1,\ldots, 2n_{0}^{2}$ there exists
$s_{i}^{k}\in\mathcal{P}_{k}$ such that
$s_{i}^{k}<\alpha_{k}<\beta_{k}<s_{i+1}^{k}$. This follows from
conditions $(d)$ and $(e)$. Set $c_{k}=s_{i}^{k}$ and
$d_{k}=s_{i+1}^{k}$.

$(ii)$ Let $\mathcal{D}=\{(c_{k},\alpha_{k}), (\beta_{k},d_{k}):
k=1,\ldots, 2n_{0}^{2}\}$. Then any $I_{1}, I_{2}$ in
$\mathcal{D}$ satisfy
\begin{center}
either $I_{1}\subset I_{2}$ or $I_{2}\subset I_{1}$ or $I_{1}\cap
I_{2}=\emptyset$\,\,.
\end{center}

$(iii)$ For all $k=1,\ldots,2n_{0}^{2}$,
\begin{align*}
\vert g_{2n_{0}^{2}}&(c_{k})-g_{2n_{0}^{2}}(\alpha_{k})\vert=\\
&\vert\sum_{j=1}^{2n_{0}^{2}}(g_{j}-g_{j-1})(c_{k})-
\sum_{j=1}^{2n_{0}^{2}}(g_{j}-g_{j-1})(\alpha_{k})\vert\geq\\
&\vert (g_{k}-g_{k-1})(c_{k})- (g_{k}-g_{k-1})(\alpha_{k})\vert-
\vert\sum_{j\not= k}(g_{j}-g_{j-1})(c_{k})- \sum_{j\not=
k}(g_{j}-g_{j-1})(\alpha_{k})\vert\geq\\
&\frac{7\e}{8}-\frac{2\e}{8}>\frac{\e}{2}\,\,,
\end{align*}
\noindent and using similar reasoning
$$
\vert
g_{2n_{0}^{2}}(\beta_{k})-g_{2n_{0}^{2}}(d_{k})\vert>\frac{\e}{2}\,\,.
$$
Next we assert that there exists a family
$\mathcal{D}^{\prime}\subset\mathcal{D}$ with
$\#\mathcal{D}^{\prime}\geq n_{0}$ consisting of pairwise disjoint
intervals.

If such a family $\mathcal{D}^{\prime}$ exists then the choice of
$n_{0}$ and relation $(iii)$ yield that $\Vert
g_{2n_{0}^{2}}\Vert_{V_{X}}\geq 2C$, which proves the claim and
completes the proof of the lemma.

Hence consider $\mathcal{D}$ and assume that any
$\mathcal{D}^{\prime}\subset\mathcal{D}$ consisting or pairwise
disjoint intervals satisfies $\#\mathcal{D}^{\prime}<n_{0}.$
Clearly $\mathcal{D}$ with the order of inclusion (i.e $I_{1}\prec
I_{2}$ iff $I_{1}\supset I_{2}$) defines a finite tree and from
our assumption each level has less than $n_{0}$ elements, while
$\#\mathcal{D}>2n_{0}^{2}$. Hence there exists a branch
\{$I_{1}\supsetneqq I_{2}\supsetneqq\ldots\supsetneqq
I_{2n_{0}}\}$ of $\mathcal{D}$ and we may assume that there exists
$\{I_{k_{1}}\supsetneqq I_{k_{2}}\supsetneqq\ldots\supsetneqq
I_{k_{n_{0}}}\}$ such that $I_{k_{i}}=(c_{k_{i}},\alpha_{k_{i}})$
for all $i\leq n_{0}$, or $I_{k_{i}}=(\beta_{k_{i}},d_{k_{i}})$
for all $i\leq n_{0}$. Then it is easy to see that if the first
alternative holds the family $\{(\beta_{k_{i}},d_{k_{i}}): i\leq
n_{0}\}$ consists of pairwise disjoint intervals. A similar
conclusion holds  if the second alternative occurs. This proves
our assertion and the proof of the lemma is complete.
\end{proof}
This lemma yields the following.
\begin{proposition}\label{p626}
Let $W$ be a bounded subset of $V_{X}^{0}$ and
$\delta>0$ such that for every relatively weakly open $U\subset W$
we have that $diam(U)>\delta.$ Then
$$
\overline{W}^{*}\cap C[0,1]\cap(V_{X}\setminus
V_{X}^{0})\not=\emptyset\,\,.
$$
\end{proposition}
\begin{proof} Applying inductively Lemma~\ref{l625} we obtain
$(g_{n})_{n}\subset W$ such that
\begin{enumerate}
\item $\Vert g_{n}-g_{n+1}\Vert_{V_{X}^{0}}>\frac{\delta}{2}$.
\item $\Vert g_{n}-g_{n+1}\Vert_{\infty}<\frac{1}{2^{n}}$.
\end{enumerate}
Hence $(g_{n})_{n}$ converges uniformly to a continuous function
$g$. Also $(g_{n})_{n}$ converges $weak^{*}$ to the same function
in the space $V_{X}$. Standard perturbation arguments yield that
$(g_{n})_{n}$ could be chosen such that the limit function $g\in
V_{X}\setminus V_{X}^{0}$. This completes the proof.
\end{proof}
As a consequence we obtain the following.
\begin{theorem}\label{t626}
Let $Y$ be a subspace of $V_{X}^{0}$. The
following are equivalent.
\begin{enumerate}
\item $c_{0}$ does not embed into $Y$.
\item The space
$Y$ has the Point of Continuity Property $(PCP)$
\end{enumerate}
In particular if $c_{0}$ does not embed into $V_{X}^{0}$, then
$V_{X}^{0}$ has $PCP$.
\end{theorem}
\section{Remarks and Problems.}
We present some problems related to the spaces $JF_{X}(\Omega)$
and certain remarks related to these problems.
\begin{problem} Let $\Omega$ be an open bounded subset of
$\mathbb{R}^{d_{0}}$, $d_{0}>1$. For $Q(\Omega)$ a family of
convex bodies contained in $\Omega$ we consider the space
$JF(Q(\Omega))$ endowed with the following norm
$$
\Vert
f\Vert_{JF(Q(\Omega))}=\sup\Bigl\{\left(\sum_{i=1}^{n}(\int_{V_{i}}
f)^{2}\right)^{1/2}: \{V_{i}\}_{i=1}^{n}\subset (Q(\Omega)),
V_{i}\cap V_{j}=\emptyset\Bigl\}\,\,.
$$
For what families $Q(\Omega)$ the following hold:

$(i)$ The space $JF(Q(\Omega))$ does not contains $\ell_{1}$.

$(ii)$ The space $JF^{*}(Q(\Omega))$ is non-separable.
\end{problem}
\begin{remark} Our proof for the family  of the
paralellepipeds $\mathcal{P}(\Omega)$ does not yield similar
results for other families. In particular the above problem is
open if $Q(\Omega)$ is either the Euclidean balls or the
$\ell_{\infty}^{d_{0}}$ balls contained in $\Omega$.
\end{remark}
\begin{problem} Suppose that $\Omega$, $\Omega^{\prime}$ are open
and bounded subsets of $\mathbb{R}^{d_{0}}$. Is it true that
$JF_{X}(\Omega)$ is isomorphic to $JF_{X}(\Omega^{\prime})$.
\end{problem}
\begin{remark} As we have mentioned at the beginning of the
first section $JF_{X}((0,1)^{d_{0}})$ is isomorphic to a
complemented subspace of $JF_{X}(\Omega)$ for any open
$\Omega\subset\mathbb{R}^{d_{0}}$.
\end{remark}
\begin{problem} Is it possible for $1<d_{1}<d_{0}$ the space
$JF_{X}((0,1)^{d_{0}})$ be isomorphic to a subspace of
$JF_{X}((0,1)^{d_{1}})$.
\end{problem}
Corollary~\ref{c37} yields that this is not possible if $d_{1}=1$.
But the argument used for this result is not extended in higher
dimensions.
\begin{problem}
Does $c_{0}$ embed into $JF_{X}$ for any $X$ reflexive with
1-symmetric basis.
\end{problem}
This problem is related to our results presented  in the last two
sections. There are two ways to approach a positive answer to this
problem. The first is the following
\begin{question}
Does there exist a property similar to $CCP$ valid in any
reflexive space $X$ with 1-symmetric basis which implies the
existence of a sequence in $JF_{X}$ equivalent to $c_{0}$ basis.
\end{question}
\noindent The second concerns the following which summarize some
of our results from section 6.
\begin{ttheorem} The following are
equivalent.
\begin{enumerate}
\item $c_{0}$ embeds into $JF_{X}$.
\item $C[0,1]\cap (V_{X}\setminus V_{X}^{0})\not=\emptyset$.
\item $V_{X}^{0}$ fails $PCP$.
\item The identity $I:\; V_{X}^{0}\to C[0,1]$ is not
semi-embedding.
\end{enumerate}
\end{ttheorem}
Hence a second approach is to show that some of the above
equivalents holds for any $JF_{X}$ space.

The last problem concerns the structure of $JF$.
\begin{problem}
Does every subspace of $JF$ contains either $c_{0}$ or $\ell_{p}$,
$2\leq p<\infty$.
\end{problem}
As we have mentioned in the introduction the space $JF$ contains
$\ell_p$ for $2\le p<\infty$ (\cite{Bu}). From the results of
section 6, follows that this problem is reduced to the case of
subspaces $Y$ of $V_{X}^{0}$ which are reflexive and the identity
$I:\;Y\to C[0,1]$ is a compact operator.

\end{document}